\documentclass[reqno, 12pt, reqno]{amsart}

\usepackage{amsfonts, amsthm, amsmath, amssymb}
\usepackage{hyperref}
\hypersetup{colorlinks=false}

\usepackage[margin=3cm]{geometry}

\RequirePackage{mathrsfs} \let\mathcal\mathscr

\numberwithin{equation}{section}

\newtheorem{theorem}{Theorem}[section]
\newtheorem{lemma}[theorem]{Lemma}
\newtheorem{proposition}[theorem]{Proposition}

\newtheorem{conjecture}[theorem]{Conjecture}
\newtheorem{prop}[theorem]{Proposition}
\newtheorem{cor}[theorem]{Corollary}

\theoremstyle{definition}

\newtheorem{remark}[theorem]{Remark}
\newtheorem{definition}[theorem]{Definition}

\renewcommand{\d}{\mathrm{d}}
\renewcommand{\phi}{\varphi}
\renewcommand{\rho}{\varrho}

\newcommand{\PP}{\mathbb{P}}
\renewcommand{\AA}{\mathbb{A}}

\newcommand{\FF}{\mathbb{F}}
\newcommand{\ZZ}{\mathbb{Z}}

\newcommand{\NN}{\mathbb{N}}
\newcommand{\QQ}{\mathbb{Q}}
\newcommand{\RR}{\mathbb{R}}
\newcommand{\CC}{\mathbb{C}}

\newcommand{\TT}{\mathbb{T}}

\newcommand{\bell}{\boldsymbol{\ell}}

\renewcommand{\leq}{\leqslant}

\renewcommand{\geq}{\geqslant}

\renewcommand{\bar}{\overline}

\newcommand{\h}{\mathbf{h}}

\newcommand{\g}{\mathbf{g}}

\newcommand{\x}{\mathbf{x}}
\newcommand{\y}{\mathbf{y}}

\renewcommand{\u}{\mathbf{u}}
\newcommand{\z}{\mathbf{z}}

\renewcommand{\a}{\mathbf{a}}

\newcommand{\ve}{\varepsilon}

\DeclareMathOperator{\mor}{Mor}

\DeclareMathOperator{\Conf}{Conf}

\def \Mor {\operatorname{Mor}}
\def \Ql {\mathbb Q_{\ell}}
\def \Pconf {\operatorname{PConf}}
\def \Conf {\operatorname{Conf}}
\def \mcL {\mathcal L}
\def \red {\operatorname{red}}
\def \chaar {\operatorname{char}}
\def \sgn {\operatorname{sgn}}

\usepackage[usenames,dvipsnames]{color}

\usepackage{tikz-cd}

\begin{document}

\title{A geometric version of the circle method}

\author{Tim Browning}
\address{IST Austria\\
Am Campus 1\\
3400 Klosterneuburg\\
Austria}
\email{tdb@ist.ac.at}

\author{Will Sawin}
\address{Columbia University\\  
Department of Mathematics\\
2990 Broadway\\ New York\\ NY 10027\\ USA}
\email{sawin@math.columbia.edu}


\begin{abstract}
We develop a geometric version of the circle method and use it to compute the compactly supported cohomology of the space  
of rational curves through a point on a smooth affine hypersurface of sufficiently low degree. 
\end{abstract}

\date{\today}

\maketitle

\thispagestyle{empty}
\setcounter{tocdepth}{1}
\tableofcontents

\section{Introduction}\label{s:intro}

Let $n\geq 2$ and  let $X \subset \mathbb A^n$ be a smooth hypersurface defined by a polynomial $f(x_1,\dots,x_n)$ of degree $k\geq 3$ with coefficients in a field $K$. 
Assume that the leading terms of $f$  
define a smooth projective hypersurface $Z\subset \mathbb P^{n-1}$ of
degree $k$.
For a given point  $P=(x_1: \ldots :x_n)\in Z(K)$, the  goal of this paper is to study the space $\Mor_{d,P}(\mathbb A^1, X)$ of  
$n$-tuples of polynomials $g_1,\dots,g_n$ of degree $d$ that satisfy $f(g_1,\dots,g_n)=0$ and whose leading 
coefficients
are exactly $(x_1,\dots,x_n)$.
The space $\mor_{d,P}(\AA^1,X)$ is cut out by $dk$ equations in $dn$ variables and so its expected dimension   is $d(n-k)$.  
One of the outcomes of our work will be a proof of this fact for $X$ of sufficiently low degree. 

There has been a lot of recent work directed at  studying 
the Kontsevich moduli space 
$\overline{\mathcal{M}}_{0,0}(Z,d)$ of
degree $d$ rational curves on degree $k$ hypersurfaces $Z\subset \PP^{n-1}$. 
Building on pioneering work of 
 Harris, Roth and Starr \cite{HRS}, 
Riedl and Yang \cite{RY} have proved that 
$\overline{\mathcal{M}}_{0,0}(Z,d)$
 is an irreducible, local complete intersection scheme of the expected dimension, provided that $Z$
is general and $n\geq k+3$. 
This can be extended to all smooth hypersurfaces when $k=3$, 
thanks to work of  Coskun and Starr \cite{CS}.
Finally, for 
$n> 2^{k-1}(5k-4)$, Browning and Vishe \cite{BV'} 
have adapted the Hardy--Littlewood circle method to handle the space of 
degree $d$ rational curves on   arbitrary  smooth hypersurfaces of degree $k$ in $\PP^{n-1}$.

In the present investigation we wish to go  further and examine the  
compactly supported cohomology of the space  
of rational curves on a  smooth hypersurface.  It turns out that it is technically easier to work with affine space and to fix the point at infinity. 
This has led us to focus our present efforts on understanding $\Mor_{d,P}(\mathbb A^1, X)$, for $X \subset \mathbb A^n$ a smooth hypersurface over a field $K$, whose leading terms 
define a smooth projective hypersurface $Z\subset \mathbb P^{n-1}$, with 
 $P\in Z(K)$.

Let  $\Pconf_m$ be the pure configuration space parametrising ordered $m$-tuples of distinct points in $\AA^1$, and let $\Conf_m$ be the configuration space parametrising unordered $m$-tuples. 
The space $\Pconf_m$ carries a free action of 
 the symmetric group $S_m$, the quotient by which is $\Conf_m$.
The cohomology  of these spaces has been studied by 
Arnol'd \cite{arnold}, a topic that has been revisited and connected to number theory over function fields by 
Chuch, Ellenberg and Farb \cite{CEF}.
Let $\sgn$ be the sign representation of $S_m$
and observe that  $\sgn^{n-1}$ is the trivial representation when $X$ is even dimensional.
We are now ready to reveal our  main result.

 \begin{theorem}\label{main2}   
 Let $\ell$ be a prime.
 If  $\chaar(K) \neq 0$ assume that   $\chaar(K)>k$ and  $\ell$ has even order modulo $\chaar(K)$. 
Assume that $d\geq k-1\geq 2$ 
and $n > 2^{k}(k-1) $.  
 There exists a spectral sequence $E^{m,s}_r$, whose first page  $E^{m,s}_1$ is
 $$
 \begin{cases}
 ( H^{s+mn  }_c (\Pconf_m,\Ql) \otimes H^{n-1}_c(X_{\overline{K}},\Ql)^{\otimes m}  \otimes \sgn^{n-1} )^{S_m}  & \text{ if $0\leq m\leq d$,}\\
 0 &\text{ otherwise,}
 \end{cases}
 $$ 
and where the differential $d_r^{m,s}:E_r^{m,s}\to E_{r}^{m+r,s-r+1}$ 
 has bidegree $(r,1-r)$, which converges to a complex whose $i$th cohomology is isomorphic as a $\mathbb Q_\ell$-vector space to $H^{i+2d(n-k)}_c( \Mor_{d,P}(\mathbb A^1,X)_{\overline{K}},\Ql)$, provided that 
$$
 i>-
 4\left( \left\lfloor \frac{d}{k-1}\right\rfloor \left(\frac{n}{2^{k}}-k+1\right)-1\right).
$$ 
Furthermore,  $d_r^{m,s}$ vanishes on every page $E_r^{m,s}$ such that  $r$  is odd.
\end{theorem} 

Note that  $
 4\left(\left\lfloor \frac{d}{k-1}\right\rfloor \left(\frac{n}{2^{k}}-k+1\right)-1\right)
\geq 0$
if and only if
\begin{equation}\label{eq:n-range}
\left\lfloor \frac{d}{k-1}\right\rfloor \left(\frac{n}{2^{k}}-k+1\right)\geq 1.
\end{equation}
Since 
  $d\geq k-1$, it suffices to have 
$n \geq 2^k k$ to ensure that \eqref{eq:n-range} holds.
Among other things, we shall use Theorem \ref{main2} to prove that 
\begin{equation}\label{eq:pen}
i>0 \quad \Longrightarrow \quad H^{i+2d(n-k)}_c(\Mor_{d,P}(\mathbb A^1,X)_{\overline{K}},\Ql)=0 ,
\end{equation}
provided that $K,\ell$ satisfy the hypotheses of the theorem, with 
 $d\geq k-1\geq 2$ and \eqref{eq:n-range} holding.
This is compatible with the expectation that $\Mor_{d,P}(\mathbb A^1,X)$ has dimension $d(n-k)$.

A few words are in order concerning the ranges for the indices appearing in the spectral sequence in Theorem \ref{main2}.
We will show later (in Lemma \ref{proof-support}) that  
\begin{equation}\label{eq:support}
E_1^{m,s}\neq 0 \quad\Longrightarrow \quad 
\begin{cases}
(m,s)=(0,0), \text{ or }\\
m \in [1,d] \text{ and } s \in [-m(n-1)+1 ,-m(n-2)].
\end{cases}
\end{equation}  
In particular the spectral sequence is supported in the quadrant where $s\leq 0$ and $m\geq 0$, and is bounded for fixed $d$. See  Figure \ref{fig:spec} for a depiction
of part of the support of $E_1^{m,s}$ when $n=4$,
which is triangular in shape. 
  
Using \eqref{eq:support}, we note that 
if $E^{m,s}_1$ vanishes for any $m,s$, then  $E^{m,s}_\infty$ also has to vanish. By construction, the $i$th cohomology of the complex that $E^{m,s}_1$ converges to is an iterated extension of $E^{m,s}_{\infty}$ for $m+s =i$. This therefore
leads to the
following consequence of Theorem~\ref{main2},
 which clearly implies \eqref{eq:pen}.

 \begin{figure} 
\[ 
\begin{footnotesize}
\begin{tikzcd} & m=0 & m=1 &  m=2 &m=3  & [-10pt] m=4\\
s=0 & \Ql \arrow[r, "d^{0,0}_1"] \arrow[drr, dashed, "d^{0,0}_2"] \arrow[ddrrr, swap, dashed, "d^{0,0}_3"]& 0 & 0  & 0 & 0 \\
s=-1 & 0 & 0 & 0  & 0 & 0 \\
s=-2 &0  & H 
& 0 & 0 & 0 \\
s=-3 & 0 & 0 & 0 & 0 & 0 \\
s=-4 & 0& 0 & \bigwedge^2 H 
& 0 & 0  \\
s= -5 & 0  & 0 & \bigwedge^2 H
& 0 & 0 \\
s=-6 & 0 & 0 & 0 & \bigwedge^3 H
& 0 \\
s= -7 & 0  & 0& 0 & H \otimes \bigwedge^2 H  
& 0 \\ 
s= -8& 0 & 0& 0 &  \left(H^{ \otimes 3}  \otimes \operatorname{std}\right)^{S_3} & \bigwedge^4 H  \\
s= -9 & 0  & 0  & 0  & 0  & \bigwedge^2 H \otimes \bigwedge^2 H
 \end{tikzcd}
 \end{footnotesize}
 \] 
 \caption{This figure illustrates part of our spectral sequence when $n=4$, where 
 $H=H^{n-1}_c(X_{\overline{K}}, \Ql)$ and $\operatorname{std}$ denotes the standard two-dimensional representation of $S_3$. The dashed arrows denote differentials that will appear on later pages of the spectral sequence. 
The entries are obtained via a straightforward calculation of the $S_m$-action on the cohomology of the configuration spaces from, for which one can use  the description of \cite[Thm.~4.5]{LehrerSolomon} or the presentation of Arnol'd \cite{arnold}. For general $n$, the non-zero terms form a triangle that extends downward and rightward from the point $(-(n-2),1)$. One can see that in this region of the spectral sequence, there are very few differentials that are potentially non-zero.
  The diagonal lines with constant $m+s$ in our figure will contribute to different cohomology groups. } \label{fig:spec}
\end{figure}

 \begin{cor}\label{cor:stable}
  Let $\ell$ be a prime. 
If  $\chaar(K) \neq 0$ assume that   $\chaar(K)>k$ and  $\ell$ has even order modulo $\chaar(K)$.   
Assume that $d\geq k-1\geq 2$
and $n > 2^{k}(k-1) $.  
If $ 
 i>-
 4\left(\left\lfloor \frac{d}{k-1}\right\rfloor \left(\frac{n}{2^{k}}-k+1\right)-1\right)$ 
 then   
$$
H^{i+2d(n-k)}_c(\Mor_{d,P}(\mathbb A^1,X)_{\overline{K}},\Ql)=0
$$
unless $i=0$ or  $i \in [1 -m(n-2),-m(n-3)]$ for some $m \in [1,d]$. 
\end{cor}

Our proof of Theorem~\ref{main2}  
uses  ``spreading out'', in the sense of Grothendieck \cite[\S~10.4.11]{EGAIV},
which transfers us to the analogous problem  over 
the algebraic closure of a finite field.  In fact when $K$ is a finite field we give a precise description of the spectral sequence in Theorem~\ref{main}, replete with information about the Galois action through Tate twists. 

The key innovation in this paper is the introduction of a  geometric analogue of the circle method, which is inspired by the sorting of exponential 
sums according to ``major arcs'' and ``minor arcs'' that is found in  the usual Hardy--Littlewood circle method.   The treatment of the major arcs is entirely geometric, but guided by the kind of calculations that occur in the circle method. The treatment of the minor arcs, on the other hand,  reduces to a point counting problem over finite fields. This will be reinterpreted as a point counting problem over the function field $\FF_q(T)$, to which existing circle method techniques
developed by  Lee \cite{lee}  and  Browning--Vishe \cite{BV,BV'}
can be adapted. (A different approach to the circle method over function fields was developed by Pugin~\cite{pugin}).

\medskip

A sequence of spaces $Y_d$ is said to be 
``homologically stable'' if the $i$th cohomology of $Y_d$ is independent of $d$ for $d\gg i$. The precise meaning of this in our context is not immediately clear since there is no natural map 
$\Mor_{d,P}(\mathbb A^1,X)\to \Mor_{d+1,P}(\mathbb A^1,X)$. 
However, 
assuming that the differentials  in the spectral sequence of Theorem \ref{main2} are independent of $d$ for $d$ sufficiently large, it follows that  there exist isomorphisms 
$$H^{i+2d(n-k)}_c( \Mor_{d,P}(\mathbb A^1,X)_{\overline{K}},\Ql) \cong H^{i+2(d+1)(n-k)}_c( \Mor_{d+1,P}(\mathbb A^1,X)_{\overline{K}},\Ql)$$ for $d$ sufficiently large. Thus a form of homological stability applies to the sequence of spaces $\Mor_{d,P}(\mathbb A^1,X)$.

There have been two recent success stories where results in analytic number theory have been established by proving homological stability of appropriate moduli spaces.
In work of Ellenberg, Venkatesh and Westerland \cite{EVW}, 
a  homological stabilisation theorem is established  for the moduli space of branched covers of the complex projective line.  
For a given odd prime $l$ and a finite abelian $l$-group $A$, this is used to prove 
that (for sufficiently large $q\not \equiv 1\bmod{l}$) 
a positive proportion of quadratic extensions of $\FF_q(T)$ have the $l$-part of their class group isomorphic to $A$. This point of view has been taken even further by 
Ellenberg, Tran and Westerland \cite{ETW}, 
where a similar philosophy is used to 
confirm the  upper bound in  Malle's conjecture about the distribution of finite extensions of 
$\FF_q(T)$ with a specified Galois group. Theorem \ref{main2} and
Corollary ~\ref{cor:stable} go  
in the reverse direction, whereby a 
homological stabilisation theorem is proved using methods which herald from analytic number theory.

We expect that the spectral sequence in Theorem~\ref{main2} 
degenerates on the first page for sufficiently large $m+s$, leading  us to make the following 
conjecture.

\begin{conjecture}\label{con:degenerate}
Assume that $d\geq k-1\geq 2$
and $n >   2^{k}(k-1) $.  
The cohomology group 
$
H^{i+2d(n-k)}_c( \Mor_{d,P}(\mathbb A^1,X)_{\bar K},\Ql)$
is isomorphic  to
$$
\bigoplus_{m \geq 0} ( H^{i+m(n-1)  }_c (\Pconf_m,\Ql) \otimes H^{n-1}_c(X_{\bar K},\Ql)^{\otimes m} \otimes \sgn^{n-1} )^{S_m},
$$ 
for 
$
 i>
 -4\left(\left\lfloor \frac{d}{k-1}\right\rfloor \left(\frac{n}{2^{k}}-k+1\right)-1\right)$.
\end{conjecture}

When the homology of a sequence of spaces stabilises, it is natural to find a single space that they all map to, whose homology is the limit of the homology of the underlying spaces. (When each space maps to the next space in the sequence, one can simply take the limit of the spaces.) In  \S~6 we shall investigate a potential space with this property for the sequence of $\Mor_{d,P}(\AA^1,X)$ over the complex numbers.  We shall demonstrate that the various
$\Mor_{d,P}(\mathbb A^1,X)$ naturally map to the space $Hom_{d,P}(\mathbb C,X)$  that parameterises certain continuous maps  from 
$\mathbb P^1(\mathbb C)$ to the smooth projective closure $\overline{X}$ of $X$ and is homotopic to the double loop space of $X$. We conjecture a relationship between the compactly supported cohomology of  $\Mor_{d,P}(\mathbb A^1,X)$ and the cohomology of this double loop space. In Theorem~\ref{t:top} we will show that 
Conjecture~\ref{con:degenerate} follows from this conjecture.

Although we have not been able to compute the full cohomology of the space of maps 
$\Mor_{d,P}(\mathbb A^1,X)_{\bar K}$,
we can nonetheless use Theorem~\ref{main2} 
to calculate $2(n-2) (n-3) -1$ cohomology groups, starting from the first non-zero cohomology group, without higher differentials. 
To see this, 
we claim that $E^{m,s}_\infty = E^{m,s}_1$ for $m+s >1 -2 (n-2)(n-3)$,
for  the spectral sequence $E^{m,s}_r$ in Theorem ~\ref{main2}. 
This follows if we are able to show that the differentials on the $r$th page of the spectral sequence vanish for $r\geq 1$ and $m+s > -2(n-2)(n-3)$.  Suppose 
for a contradiction that there is  some differential $d^{m,s}_r: E^{m,s}_r \to E^{m+r,s-r+1}_r$ that is non-zero. The last part of 
Theorem~\ref{main2}  allows us to assume that $r\geq 2$.  Then certainly $E^{m,s}_1$ and $E^{m+r,s-r+1}_1$ must be non-zero, which by
\eqref{eq:support}
can only occur when the interval $[  -m(n-1)+1 ,  -m (n-2) ]$  intersects 
the interval
$$[  -(m+r)(n-1) + r , -(m+r) (n-2) +r-1].
$$ 
This requires $m(n-1)-1 \geq (m+r)(n-2)-r+1$, which is equivalent to  $m \geq  r(n-3) + 2$. This in turn implies  that  $m \geq 2n-4$. But 
 \eqref{eq:support} implies that 
 $s\leq -m(n-2)$ if  $ E^{m,s}_1$ is non-vanishing, 
whence $$ m+ s \leq -nm+3m = - m(n-3) \leq -2 (n-2) (n-3),
$$ 
as required.

A straightforward outcome of Theorem~\ref{main2} is the following result (proved as Corollary~\ref{cor:dim2} below), which concerns the most basic geometric properties of 
$\mor_{d,P}(\AA^1,X)$. This should be viewed as an analogue of the main theorem of \cite{BV'} in a different context, with a different, more geometric proof. We expect that a proof of this theorem exists via the methods of \cite{BV'}, and vice versa.

\begin{cor}[Corollary \ref{cor:dim2}] \label{cor:dim}
Let $K$ be a field with
 $\chaar(K)>k$ if  $\chaar(K) \neq 0$. 
 Assume that $d\geq k-1\geq 2$ and that   \eqref{eq:n-range} holds.
 Then the
space $\mor_{d,P}(\AA^1,X)$ is irreducible and has the expected dimension $d(n-k)$.
\end{cor}

\begin{remark}
It is unreasonable to expect an analogue of  
Corollary~\ref{cor:dim} when $d<k-1$.
Let $X$ be the vanishing locus of the polynomial
$$
f=x_1^{k-1}x_2+x_2^{k}+\dots+x_n^k+1.
$$
Then the leading form defines a smooth hypersurface $Z\subset \PP^{n-1}$ of degree $k$, which contains the point
$P=(1:0:\dots:0)$. 
Consider the subscheme $M\subseteq
\mor_{d,P}(\AA^1,X)$, consisting of rational curves which lie in the hyperplane $x_2=0.$ If $Y\subset \AA^{n-2}$ denotes the smooth hypersurface 
$x_3^{k}+\dots+x_n^k+1=0$, then  $\mor_{d,(0:\dots:0)}(\AA^1,Y)$ 
is cut out by $k(d-1)+1$ equations in $d(n-2)$ variables. 
Hence it has  dimension at least  $d(n-2)-k(d-1)-1$. This implies that
$$
\dim M\geq d+d(n-2)-k(d-1)-1=d(n-k)-d+k-1.
$$
Thus, 
if $d<  k-1$, then the  dimension of $ M$ is greater than the expected dimension $d(n-k)$ of $ \mor_{d,P}(\AA^1,X)$.
\end{remark}

\subsection*{Acknowledgements}
During the preparation of this  paper the authors were
supported by the NSF under Grant No.\ DMS-1440140,  while  in residence at 
the {\em 
Mathematical Sciences Research Institute} in Berkeley, California,
during the Spring 2017 semester.  
Tim Browning was supported by EPSRC grant \texttt{EP/P026710/1} and 
ERC grant \texttt{306457}, while 
Will Sawin was supported by Dr. Max R\"{o}ssler, the Walter Haefner Foundation and the ETH Z\"urich
Foundation. We would like to thank the anonymous referee for helpful comments.

\section{Overview of the argument}

To begin with we proceed under the assumption that
$K=\FF_q$ is a finite field such that 
$\chaar(\FF_q)>k$.  At the end of the section we will deduce Theorem~\ref{main2} by using a spreading out argument. We take $\ell$ to be a prime which has even order in the multiplicative group modulo  
the characteristic of $\FF_q$. This is a technical hypothesis that could, of course, be removed if strong enough independence-of-$\ell$ results for \'{e}tale cohomology were known. 
Our argument will rely mainly  on foundational results in the theory of \'{e}tale cohomology from  \cite{sga4-3}. These comprise: 
\begin{itemize}
\item 
proper base change \cite[Expos\'e XIII, Prop.~5.2.8]{sga4-3}; 
\item
 Leray spectral sequence with compact supports \cite[Expos\'e XVII, Eq.~(5.1.8.2)]{sga4-3};

\item functoriality \cite[Expos\'e XVII, Variant 5.1.14]{sga4-3}; 

\item
excision \cite[Expos\'e XVII, Eq.~(5.1.16.2)]{sga4-3};

\item the projection formula \cite[Expos\'e XVII, Prop.~5.2.9]{sga4-3}; and 
\item
the K\"{u}nneth formula \cite[Expos\'e XVII, Thm.~5.4.3]{sga4-3}. 
\end{itemize}

From now on  $X\subset \AA^n$ is a smooth hypersurface defined by  $f\in \FF_q[x_1,\dots,x_n]$ with degree $k\geq 3$, whose 
 leading terms define a smooth hypersurface $Z\subset \PP^{n-1}$ over $\FF_q$.  For $
P=(x_1:\dots: x_n)\in Z(\FF_q),
$
our interest lies with the space  $\mor_{d,P}(\AA^1,X)$  of  $n$-tuples of polynomials $g_1,\dots,g_n\in \FF_q[T]$ of degree $d$ that satisfy  $f(g_1,\dots,g_n)=0$ and whose leading coefficients are exactly $(x_1,\dots,x_n)$.
The polynomials of interest to us take the shape
\begin{equation}\label{eq:def-gj}
g_j(T)=x_j T^d +   \sum_{i=0}^{d-1} a_{i,j} T^i,
\end{equation}
for $1\leq j\leq n$ and 
$\left(a_{0,j}, \dots,a_{d-1,j}\right)_{1\leq j\leq n}\in (\FF_q^{d})^n$.
 Let $e:\left(\mathbb A^d\right)^n  \times \mathbb A^{kd}\to \AA^1$ be the function that takes 
$$
\left(
\left(a_{0,j}, \dots,a_{d-1,j}\right)_{1\leq j\leq n}
,(b_1,\dots,b_{kd})\right)
$$
to the coefficient of $T^{-1}$ in \[ \left( \sum_{r=1}^{kd} b_r T^{-r} \right) f\left( x_1 T^d +   \sum_{i=0}^{d-1} a_{i,1} T^i, \dots, x_n T^d  +   \sum_{i=0}^{d-1} a_{i,n} T^i\right).\]

It will be instructive to recall how the function field version of the classical circle method  can be used to study 
$\mor_{d,P}(\AA^1,X)$. Let $\TT$ be the ring $\{\sum_{r\geq 1} b_r T^{-r}:
b_r\in \FF_q\}$ of formal power series in $1/T$, with constant term equal to zero.
Let  $\phi:\FF_q((1/T))\to \FF_q$ be the function
which takes an element of 
$\FF_q((1/T))$ to the coefficient of $T^{-1}$ and 
let  $\psi$ be
 a non-trivial additive character of $\mathbb F_{q}$. Then  $\psi\circ \phi$ is a non-trivial  (additive) character 
on the locally compact space $\FF_q((1/T))$. 
By combining properties of the Haar measure on $\TT$ with the orthogonality of characters, we have 
\begin{equation}\label{eq:swap}
\#\mor_{d,P}(\AA^1,X) 
=  \sum_{g_1,\dots,g_n}
\int_\TT \psi\circ\phi(\alpha f(g_1,\dots,g_n)) \d \alpha=\int_{\TT}  S(\alpha)
\d \alpha,
\end{equation}
where the sum is over polynomials $g_1,\dots,g_n$ of the shape \eqref{eq:def-gj} and 
$$
S(\alpha)=
\sum_{g_1,\dots,g_n}
 \psi\circ\phi(\alpha f(g_1,\dots,g_n)) .
$$
This observation is the igniting spark in the circle method.

When $\alpha$ is close to an element of $\FF_q(T)$ with small denominator
we expect $S(\alpha)$ to be large. The union of such points form the set of  ``major arcs" and ought to make the dominant contribution to 
$\#\mor_{d,P}(\AA^1,X)$. The ``minor arcs'' constitute everything else, which one would like  to show make a negligible contribution to this cardinality. Once achieved, the sort of information  about the geometry of $\mor_{d,P}(\AA^1,X)$ embodied in Corollary~\ref{cor:dim} can be deduced from the resulting asymptotic formula by comparing the count to what is predicted by the  Lang--Weil estimate \cite{LW}.

The polynomial \[f\left( x_1 T^d +   \sum_{i=0}^{d-1} a_{i,1} T^i, \dots, x_n T^d  +   \sum_{i=0}^{d-1} a_{i,n} T^i\right)\] is a degree $k$ polynomial composed with degree $d$ polynomials and hence has degree  at most $ kd$. 
In fact,   
since $(x_1:\dots:x_n)\in Z(\FF_q)$, 
it is clearly a polynomial of degree at most $kd-1$. Thus only the piece $\sum_{r=1}^{dk}b_r T^{-r}$ plays a role in the integration over $\alpha\in \TT$ in \eqref{eq:swap} and the exponential sum 
in the circle method
can be written
$$
S(\alpha)=\sum_{a_{i,j}} \psi\left( e(a_{i,j};b_1,\dots,b_{kd})\right).
$$
We require a geometric analogue of this approach.

Let $p_1 : \left(\mathbb A^d\right)^n  \times \mathbb A^{kd} \to \left(\mathbb A^d\right)^n $ and $p_2 : \left(\mathbb A^d\right)^n  \times \mathbb A^{kd} \to \mathbb A^{kd}$ be the natural projections.
Let $\mcL_\psi$  be the  Artin--Schreier sheaf on $\mathbb A^1$ associated to the character $\psi$ on $\FF_q$.
We will work with the complex 
 $$
 S_{d,f}=R p_{2!}  e^* \mcL_\psi
 $$
 on $\mathbb A^{kd}$. 
Our geometric version of \eqref{eq:swap} involves writing the desired compactly supported cohomology group as the compactly supported cohomology of $\mathbb A^{kd}$ with coefficients  in $S_{d,f}$.  
Just as the classical approach uses a 
special case of the inversion formula for Fourier series, the following geometric analogue is proved using  a special case of the inversion formula for the $\ell$-adic Fourier transform.

\begin{lemma}\label{Fourier} 
We have 
$H^i_c (\Mor_{d,P}(\mathbb A^1, X)
_{\overline \FF_q} 
, \Ql )= H^{i+2kd}_c \left( \mathbb A^{kd}, S_{d,f}(kd)\right)$ for any $i\in \ZZ$. 
\end{lemma}

Before turning to the proof of this result, we note that  Lemma~\ref{Fourier} can be used to recover the expression \eqref{eq:swap} by taking trace functions of both sides and appealing to the Grothendieck--Lefschetz trace formula to interpret the left hand side in terms of the cardinality of  
$\mathbb F_q$-points of $\Mor_{d,P}(\mathbb A_1,X)$ and 
the right hand side in terms of exponential sums over $\FF_q$-points of
$\mathbb A^{kd}$.

In Definition~\ref{def:major}
we shall define a certain Zariski closed subset $A_{d}^{kd}$ of $\AA^{kd}$, 
which plays the role of the major arcs in our geometric setting. 
The geometric analogue of the estimation of the integral over the major arcs should be a computation of the cohomology of $A_d^{kd}$ with coefficients in the complex $S_{d,f}$. Unfortunately, we are only able to  compute the cohomology using a spectral sequence and evaluate the first page.  We are not able to rule out the existence of higher differentials causing cancellation in the cohomology, a phenomenon which is invisible on the trace function side. Nonetheless, we are still able to prove strong upper bounds on the dimension of the cohomology of $A_{d}^{kd}$.  The geometric analogue of the bound on the minor arcs $\AA^{kd}-A_d^{kd}$ will be a cohomology vanishing statement. 

 \begin{proof}[Proof of Lemma~\ref{Fourier}]
Let $c_1,\dots,c_{kd}$ be the unique polynomials in the $a_{i,j}$ such that 
  \[f\left( x_1 T^d +   \sum_{i=0}^{d-1} a_{i,1} T^i, \dots, x_n T^d  +   \sum_{i=0}^{d-1} a_{i,n} T^i\right)= \sum_{r=1}^{kd} c_r T^{r-1}.
  \]
  Then  $e(a_{i,j};b_1,\dots,b_{kd})=
 \sum_{r=1}^{kd} b_r c_r$. 
Let $\mu: \mathbb A^{kd} \times \mathbb A^{kd} \to \mathbb A^1$ be the map 
$$
\left((b_1,\dots,b_{kd}), (c_1,\dots,c_{kd})\right)\mapsto \sum_{r=1}^{kd} b_r c_r.
$$
Then $e = \mu \circ (c \times id)$.
Furthermore, $\Mor_{d,P}(\mathbb A^1,X)
_{\overline \FF_q} 
$ is the fiber of $c$ 
over the point $0$, since we can view $(\mathbb A^d)^n$ as the space of $n$-tuples of degree $d$ polynomials $(g_1,\dots,g_n)$, with 
$g_j = x_jT^d+ \sum_{i=0}^{d-1} a_{i,j} T^j$,  
whose leading coefficients are exactly $(x_1,\dots,x_n)$. Moreover, in  that space the condition $f(g_1,\dots,g_n)=0$ is precisely the condition that the image of the point under $c$ is zero.
 
By the Leray spectral sequence with compact supports, we have   
\begin{align*}
H^{i+2kd}_c \left( \mathbb A^{kd}, S_{d,f}(kd)\right)
=~&  H^{i+2kd}_c \left( \mathbb A^{kd}, R p_{2!}  e^* \mcL_\psi(kd)\right)\\
=~& H^{i+2kd}_c\left( \left(\mathbb A^d\right)^n  \times \mathbb A^{kd} , e^* \mcL_\psi(kd)\right) \\
=~& H^{i+2kd}_c\left( \left(\mathbb A^d\right)^n  \times \mathbb A^{kd} , (c \times id)^* \mu^* \mcL_\psi(kd)\right)\\
=~& H^{i+2kd}_c \left( \mathbb A^{kd} \times \mathbb A^{kd}, R(c \times id)_! \Ql \otimes \mu^* \mcL_\psi(kd) \right).
\end{align*}
The desired statement follows from this identity and a special case of \cite[Lemma 13]{betti}. Katz's proof is a sketch and so we give more detail here for the sake of  completeness.

If we let $p_1': \mathbb A^{kd} \times \mathbb A^{kd} \to \mathbb A^{kd}$ be the first projection, then by proper base change we have
\begin{align*}
H^{i+2kd}_c &\left( \mathbb A^{kd} \times \mathbb A^{kd}, p_1'^* Rc_! \Ql \otimes \mu^* \mcL_\psi (kd) \right)\\
&= H^{i+2kd}_c \left( \mathbb A^{kd}, Rc_! \Ql \otimes Rp_{1!}' \mu^* \mcL_\psi (kd)\right). \end{align*}
Now $Rp_{1!}' \mu^* \mcL_\psi (kd)$ is the $\ell$-adic Fourier transform of the constant sheaf, which is a skyscraper sheaf supported at $0$ and placed in degree $2kd$. This can be checked explicitly using proper base change, which shows that its stalk at any non-zero point is the compactly supported cohomology of $\mathbb A^{kd}$ with coefficients in $\mathcal L_\psi$ of a non-constant linear map. This vanishes by the K\"{u}nneth formula, since the cohomology of $\mathbb A^1$ with coefficients in $\mathcal L_\psi$ of a non-constant linear map already vanishes. Hence  $Rp_{1!}' \mu^* \mcL_\psi (kd)$ is a skyscraper sheaf supported at $0$. Proper base change furthermore implies that the stalk at $0$ is the compactly supported cohomology of $\mathbb A^{kd}$ with coefficients $\Ql(kd)$, which is a copy of $\Ql$ in degree $2kd$.

Thus the tensor product of $Rp_{1!}' \mu^* \mcL_\psi (kd)$ with $Rc_! \Ql$ is simply the stalk of $Rc_! \Ql$ at $0$, placed at $0$ and shifted $2kd$ degrees. Hence the compactly supported cohomology in degree $i+2kd$ is simply the stalk of $R^{i}c_! \Ql$ at $0$, which by proper base change is the $i$th compactly supported cohomology of the fiber $\Mor_{d,P}(\mathbb A^1,X)
_{\overline \FF_q} 
$ of $c$ over $0$.
\end{proof}

 \begin{remark} 
 The stalk of   $S_{d,f}$ at $0$ in $\mathbb A^{kd}$ is simply the compactly 
supported cohomology of $\mathbb A^{nd}$  with coefficients in the constant sheaf $\QQ_\ell$, which vanishes outside degree $2nd$ and is one-dimensional in degree $2nd$. If the stalks of $S_{d,f}$ vanished everywhere else, then $H^{i+2kd}_c \left( \mathbb A^{kd}, S_{d,f}(kd)\right)$ would vanish for $i\neq 2d(n-k)$ and would be one-dimensional for $i=2d(n-k)$. In fact $S_{d,f}$ does not usually vanish everywhere else, but we will see later (under suitable hypotheses) that the contributions of the other points to the cohomology are in lesser degree than $2nd$, so the top degree cohomology group occurs in $2d(n-k)$. This fact is sufficient to verify that $\Mor_{d,P}(\mathbb A^1,X)$ is a variety of dimension $d(n-k)$, as in the proof of Corollary~\ref{cor:dim}. 
  \end{remark}

 It is now time to introduce the ``major arcs of level $m$'' for our geometric version of the circle method.  
 
 \begin{definition}[Major arcs]\label{def:major}
 For any integer $m\geq 0$, 
let $A_m^{kd}$ be the locus in $\mathbb A^{kd} $ consisting of points $(b_1,\dots,b_{kd})$ where the $(kd-m) \times (m+1)$ matrix $M$, whose entries are given by the formula $M_{ij}= b_{i+j-1}$, has rank at most $m$.  Take $A_{-1}^{kd} $ to be the empty set.
 \end{definition}

 We will view rational functions in $T$ as power series in $T^{-1}$. We say a power series in $T^{-1}$ is $O(T^{N})$ if it only has terms of degree at most $N$.
 The following result gives an explicit description of the major arcs of level $m$.
 
 \begin{lemma}\label{majorarcproperties}
 The set $A_{m}^{kd}$ satisfies the following properties.
  \begin{enumerate}
 
 \item $A_m^{kd}$ is a Zariski closed subset of $\mathbb A^{kd}$.
 
 \item $A_{m-1}^{kd} \subseteq A_m^{kd}$.
 
 \item $A_{m}^{kd} = \mathbb A^{kd}$ if $m\geq kd/2$.
 
 \item A tuple $(b_1,\dots,b_{kd})$ is in $ A_m^{kd}$ if and only if there exists $m'\leq m$, a polynomial $h_1(T)$ of degree $<m'$, and a monic polynomial $h_2(T)$ of degree $m'$ such that \[ \sum_{r =1}^{kd} b_r T^{-r}  = \frac{h_1(T)}{h_2(T)} + O (T^{-kd-1 + m-m' }).\]
 
 \item Assume that  $m \leq kd/2$. Then for each $(b_1,\dots,b_{kd}) \in A_m^{kd} - A_{m-1}^{kd}$, there exists a unique $m',h_1,h_2$ satisfying the conditions of part (4). Furthermore, for such $m',h_1,h_2$, the polynomials  $h_1, h_2$ are coprime and, if $m'<m$, then the coefficient of $T^{-kd-1+m-m'}$ in \[ \sum_{r =1}^{kd} b_r T^{-r}  - \frac{h_1(T)}{h_2(T)} \] is non-zero.
 \end{enumerate}
 
\end{lemma}
 
  \begin{proof}
We begin by dealing with parts (1)--(3).
Part (1) follows from the definition and the fact that the set of matrices of rank $\leq m$ is Zariski closed.
Part (2) will follow from part (4), and part (3) follows on noting that in this
 case, $M$ has $kd-m \leq m$ rows and so its rank is necessarily $\leq m$.
 
To deal with part (4)
we note that the matrix $M$ has rank $\leq m$ if and only if there is an element in its kernel. 
Suppose that an element $(c_1,\dots,c_{m+1})$ is in the kernel of $M$. 
Then for all $j\in \{1,\dots,kd-m\}$ we have 
$$\sum_{i=1}^m c_i b_{i+j-1}=0.
$$ 
In other words, the coefficient of $T^{-j}$ in $( \sum_{r=1}^{kd} b_r T^{-r} ) ( \sum_{i=1}^{m+1} c_i T^{i-1} )$ vanishes for all 
 $j\in \{1,\dots,kd-m\}$.
Let 
$$h_2(T) =  \sum_{i=1}^{m+1} c_i T^{i-1}.
$$   
and let $h_1(T)$ consist of  terms of $( \sum_{r=1}^{kd} b_r T^{-r} )h_2(T)$  of non-negative degree. 
Let $m'\leq m$ be the degree of $h_2$. Then 
$( \sum_{r=1}^{kd} b_r T^{-r} ) h_2(T)$ is a product of a Laurent series with all terms in negative degrees with a polynomial of degree $m'$. Thus it is a Laurent series with all terms in degree $<m'$. Because of the aforementioned vanishing, we have \[ \left( \sum_{r=1}^{kd} b_r T^{-r} \right) h_2(T) = h_1(T) + O(T^{-(kd-m)-1})\] and dividing both sides by $h_2(T)$ we get the desired identity.
 The converse can be proved by the same argument in reverse. Given  $h_1$, $h_2$, one multiplies the identity by $h_2(T)$, observes the vanishing of the coefficients of \[\left( \sum_{r=1}^{kd} b_r T^{-r} \right) h_2(T)\] and concludes that a vector defined by the coefficients of $h_2(T)$ lies in the kernel of $M$.

To prove part (5), we note that if $h_1$ and $h_2$ are not relatively prime, we may remove a common factor from them, decrease $m'$ by the degree of that factor, and then decrease $m$ by the same amount, to show that the point lies in $A_{m-1}^{kd}$.  Similarly if the coefficient of $T^{-kd-1+m-m'}$ in $\sum_{r =1}^{kd} b_r T^{-r}  - \frac{h_1(T)}{h_2(T)}$ is zero, then   
\[ \sum_{r =1}^{kd} b_r T^{-r}  = \frac{h_1(T)}{h_2(T)} + O (T^{-kd-2 + m-m' }).\] 
Thus we  may decrease $m$ by one and leave $m'$ fixed, again showing that the point lies in $A_{m-1}^{kd}$.
 
 It remains to show that there cannot be two distinct solutions. Let $h_1/h_2$ and $h_3/h_4$ be two distinct,  coprime solutions, with  $\deg(h_2)=m'$ and $\deg(h_4)=m''$. 
 Without loss of generality, $m' \leq m''$. Then 
 \[
 \frac{h_1(T)}{h_2(T)} - \frac{h_3(T)}{h_4(T)} 
 = O (T^{-kd-1 + m-m' }).\]
Since 
\[ \frac{h_1(T)}{h_2(T)} - \frac{h_3(T)}{h_4(T)}= \frac{h_1(T) h_4(T) - h_2(T) h_3(T) }{ h_2(T) h_4(T) }\] is non-zero, its numerator is non-zero and must have non-negative degree. Moreover,   the denominator has degree $m'+m''$ and 
so this power series has leading term in degree at least $-m'-m''$. 
It follows that  \[-m'-m'' \leq -kd-1 +m-m'. \] 
This implies that  $kd+1 \leq m + m'' \leq 2m,$
 which  contradicts the assumption that $m \leq kd/2$.
 \end{proof}
  
Ultimately, 
our analogue of the major arcs will be $A_d^{kd}$ and the analogue of the minor arcs will be its complement. We could take any cutoff $m$ and view $A_{m}^{kd}$ as the major arcs, but we will see in the next section that $m=d$ is the largest value for which the argument goes through. We shall prove the following result, which is  the geometric analogue of the estimation of the major arcs.

\begin{prop}\label{majormain} Assume that the leading terms of $f$ define a smooth hypersurface and that $\chaar(\mathbb F_{q})>k$. The spectral sequence associated to the decreasing filtration of the complex of $\ell$-adic Galois representations $H^{*}_c(A_d^{kd}, S_{d,f})$,  whose $m$th step is equal to $H^*_c(A_d^{kd} - A_{m-1} ^{kd},S_{d,f})$ for $m\in \{0,\dots,d+1\}$, has the following properties.

\begin{enumerate}
\item It converges to $H^{*}_c(A_d^{kd}, S_{d,f})$.

\item Its first page $E^{m,s} _1$ is 
$$\left( H^{s+mn -2nd }_c(\Pconf_m,\Ql) \otimes H^{n-1}_c(X
_{\overline \FF_q} 
,\Ql)^{\otimes m}  \otimes \sgn^{n-1} \right)^{S_m} (-m-n(d-m)), $$ 
whenever $0\leq m\leq d$, with  
with $E_1^{m,s}=0$ otherwise. 
\item Its differential $d_r$ has bidegree $(r,1-r)$ and it 
vanishes on every odd page. 

\end{enumerate}\end{prop}

The proof of this result will be entirely geometric, but it will follow the line of reasoning that features in the treatment of the major arcs in the usual  circle method.
The vanishing of the differentials in part (3) will be established in 
\S~\ref{s:diff}, but everything else will be achieved in 
 \S~\ref{s:major}.

The treatment of the geometric minor arcs will be the object of \S\S~\ref{s:minor1}--\ref{s:minor-2}.
It will lead to the following outcome. 

  \begin{proposition}\label{minormain}
Assume that the leading terms of $f$ define a smooth hypersurface, that $ \chaar(\mathbb F_{q})=p>k$, and that $\ell$ has even order mod $p$.  Assume that 
$d\geq k-1\geq 2$ and $ n > 2^k(k-1)$.  Then
$H^i_c( \mathbb A^{kd} - A_{d}^{kd}, S_{d,f})=0$ provided that  
$$
i>2dn+4-4\left\lfloor \frac{d}{k-1}\right\rfloor \left(\frac{n}{2^{k}}-k+1\right).
$$ 
\end{proposition}

Our investigation of the geometric major and minor arcs now leads us to draw the following conclusion.

\begin{theorem}\label{main}  Assume that the leading terms of $f$ define a smooth hypersurface in $\mathbb P^{n-1}$, that $ \chaar(\mathbb F_{q})=p>k$ and that $\ell$ has even order mod $p$.  
Assume that 
$d\geq k-1\geq 2$
 and $n >  2^{k}(k-1) $.   Let $E_r^{m,s}$ be the shift by $2dn$ of the Tate twist by $kd$ of the spectral sequence defined in Proposition~\ref{majormain}. Then the following are true. 
  
  \begin{enumerate}
  
  \item The first page $E_1^{m,s}$ of $E_r^{m,s}$ is 
  $$ ( H^{s+mn  }_c (\Pconf_m,\Ql) \otimes H^{n-1}_c(X
  _{\overline \FF_q} 
  ,\Ql)^{\otimes m} \otimes \sgn^{n-1}  )^{S_m}(kd-m-n(d-m)) $$
  whenever $0\leq m\leq d$, with  
$ E^{m,s}_1= 0$ otherwise. 

\item Its differential $d_r$ has bidegree $(r,1-r)$ and it 
vanishes on every odd page.

  \item $E_r^{m,s}$ converges to a complex whose $i$th cohomology is equal to $$H^{i+2d(n-k)}_c( \Mor_{d,P}(\mathbb A^1,X)_{\overline{\mathbb F}_q} ,\Ql),$$ as a $\Ql$-vector space with an action of $\operatorname{Frob}_q$, whenever 
\begin{equation}\label{eq:i} 
 i>-
 4\left(\left\lfloor \frac{d}{k-1}\right\rfloor \left(\frac{n}{2^{k}}-k+1\right)-1\right).
  \end{equation}  
 \end{enumerate} \end{theorem}
  
  \begin{proof} 
Parts (1) and (2) follow directly from  Proposition~\ref{majormain}. The latter result also 
implies that the spectral sequence converges to a complex whose $i$th cohomology is $H^{i+2nd}_c(A_d^{kd}, S_{d,f}(kd)).
$ 
By excision and Proposition~\ref{minormain}, if 
$$
i+2nd>
2nd+4-4\left\lfloor \frac{d}{k-1}\right\rfloor \left(\frac{n}{2^{k}}-k+1\right),
$$ 
or equivalently  \eqref{eq:i},
then the $i$th cohomology of the limit complex is isomorphic to $H^{i+2nd}_c(\mathbb A^{kd}, S_{d,f})$. 
 Finally,  
 this is  $H^{i+2d(n-k)}_c (\Mor_{d,P}(\mathbb A^1, X)_{\bar\FF_q}, \Ql )$,
 by Lemma~\ref{Fourier}, 
 which thereby completes the proof. \end{proof}
 
 To complete the proof of Theorem~\ref{main2}, we drop the assumption that $K = \FF_q$. The coefficients of $f$ form a finitely-generated $\mathbb Z$-subalgebra $R \subseteq K$. Its spectrum $\operatorname{Spec} R$ is a scheme, over which we have a family of hypersurfaces $X$, and a family of schemes $\Mor_{d,P}(\mathbb A^1, X)$. By proper base change, the cohomology groups of this family of schemes each form a constructible sheaf on $\operatorname{Spec} R$. Thus there is some open subset of $\operatorname{Spec} R$ on which each of these sheaves is lisse, and so each cohomology group is constant as a $\Ql$-vector space on this subset. Since  $\operatorname{Spec} K$ is the generic point of $\operatorname{Spec} R$, that point is contained in this open set. In view of the fact that  $R$ is finitely-generated over $\mathbb Z$, its points with finite residue fields are dense. We choose some closed point in this subset, calculate its cohomology group by a spectral sequence as in Theorem~\ref{main}, and then observe that the cohomology groups of the original $X$ are isomorphic and hence also given (noncanonically) by this spectral sequence. The only thing left is to check the conditions on the characteristic. If $K$ has a given positive  characteristic, then all the residue fields will have the same characteristic and the conditions in Theorem~\ref{main} are satisfied because we have assumed the same conditions in Theorem~\ref{main2}. If $K$ has  characteristic zero, then every open set of $\operatorname{Spec} R$ contains points of residue fields of every sufficiently large characteristic, and we can choose one of characteristic $p$ with $p>k$ and where $\ell$ has even order mod $p$ (e.g. by using quadratic reciprocity to choose $p$ so that $\ell$ is a quadratic nonresidue mod $p$).
 
 \medskip
 
 We are now ready to establish the conclusions that we drew  in  \eqref{eq:support} 
 and Corollary  \ref{cor:dim2} from the statement of  Theorem \ref{main2}.
  
 \begin{lemma}\label{proof-support} 
The first page  $E_1^{m,s}$ of  the spectral sequence in Theorem \ref{main} satisfies 
$$
E_1^{m,s}\neq 0 \quad\Longrightarrow \quad 
\begin{cases}
(m,s)=(0,0), \text{ or }\\
m \in [1,d] \text{ and } s \in [-m(n-1)+1 ,-m(n-2)].
\end{cases}
$$  
 \end{lemma}
 
 \begin{proof} Since  $ E_1^{m,s} = ( H^{s+mn  }_c (\Pconf_m,\Ql) \otimes H^{n-1}_c(X
  _{\overline \FF_q} 
  ,\Ql)^{\otimes m} \otimes \sgn^{n-1}  )^{S_m }$ for $0 \leq m \leq d$ and is zero otherwise, it vanishes unless $0 \leq m \leq d$ and $ H^{s+mn  }_c (\Pconf_m,\Ql)\neq 0$. For $m=0$, $\Pconf_m$ is a point, and has cohomology only in degree zero. Otherwises, it suffices to show that $H^{s+mn}_c(\Pconf_m,\Ql)$ vanishes unless $m+1 -mn \leq s \leq 2m-nm$, or equivalently that $H^i_c(\Pconf_m,\Ql)$ vanishes unless $m+1 \leq i \leq 2m$.
  
  To do this, we first observe  that given any configuration in $\Pconf_m$, we can translate it so that the first point is zero, letting us write $\Pconf_m $ as the product of $\mathbb A^1$ and an affine variety $\Pconf_m/\mathbb G_a$ of dimension $m-1$.  Thus, if $H^i_c( \Pconf_m , \Ql) \neq 0$, by the K\"{u}nneth formula  and the fact that $H^j_c( \mathbb G_a, \Ql)$ vanishes for $j\neq 2$ we have $H^{i-2}_c(\Pconf_m/\mathbb G_a, \Ql)\neq 0$. It follows from   Poincar\'{e} duality that 
  $$H^{2m-2 -(i-2) }(\Pconf_m/\mathbb G_a,\Ql) \neq 0,$$ 
  which is a contradiction unless $0 \leq 2m -2 - (i - 2) \leq m-1$,   because $\Pconf_m$ is affine of dimension $m-1$ \cite[Expos\'e  XIV, Cor.~3.2]{sga4-3}.  Solving for $i$, we get the desired statement.
\end{proof}

 \begin{cor}\label{cor:dim2} 
Let $K$ be a field with
 $\chaar(K)>k$ if  $\chaar(K) \neq 0$. 
 Assume that $d\geq k-1\geq 2$ and that   \eqref{eq:n-range} holds.
 Then the
space $\mor_{d,P}(\AA^1,X)$ is irreducible and has the expected dimension $d(n-k)$.
\end{cor}

\begin{proof}
For ease of notation let us write 
$M=\mor_{d,P}(\AA^1,X)$ in the proof of this result.
To begin with,  note that 
each irreducible component of 
$M$ has dimension at 
least $d(n-k)$, since 
$M$  can be defined as the vanishing locus in $(\mathbb A^d)^n$ of $dk$ equations. 

We proceed by proving that if $m$ is the dimension of the largest irreducible component of $M_{\overline{K}}$, then the cohomology group $H^i_c( M_{\overline{K}}, \Ql)$ vanishes for $i>2m$ and its dimension is equal to the number of $m$-dimensional irreducible components of $M_{\overline{K}}$ if $i=2m$. The first estimate follows from the bound for the cohomological dimension of schemes. For the second estimate, let $U$ be the maximal smooth $m$-dimensional subset of the induced reduced subscheme of $M_{\overline{K}}$ and let $Z$ be its complement. Then $U$ is smooth of dimension $m$ and its number of connected components is equal to the number of irreducible components of $M_{\overline{K}}$. Moreover,  $\dim Z\leq m-1$. Thus $H^j_c( Z_{\overline{K}},\Ql)$ vanishes for $j> 2m-2$ and the excision exact sequence yields 
\[ H^{2m-1}_c(Z_{\overline{K}},\Ql) \to H^{2m}_c(U_{\overline{K}},\Ql) \to H^{2m}_c( M_{\overline{K}} ,\Ql) \to H^{2m}_c(Z_{\overline{K}},\Ql). \] 
This therefore gives an isomorphism $H^{2m}_c(U_{\overline{K}},\Ql) \cong H^{2m}_c( M_{\overline{K}}, \Ql) $. By appealing to Poincar\'{e} duality, $H^{2m}_c(U_{\overline{K}},\Ql)$ is dual to $H^0(U_{\overline{K}},\Ql)$, which has dimension equal to the number of connected components of $U$, which is therefore  equal to the number of top-dimensional irreducible components of $M_{\overline{K}}$. 

It will follow that $M$ has dimension $d(n-k)$, with  a unique irreducible component of dimension exactly $d(n-k)$, provided we  show that 
$H^{i+ 2d(n-k)}_c( M_{\overline K},\QQ_\ell)=0$ for $i> 0$ and that it is one-dimensional for $i=0$.

When $i>0$ this follows from Corollary~\ref{cor:stable} provided that $d\geq k-1\geq 2$ and  $n$ satisfies   \eqref{eq:n-range}.
Suppose next that $i=0$
and let $E^{m,s}_r$ be a spectral sequence as in Theorem~\ref{main2}.
It follows from  \eqref{eq:n-range} that $n>3$. Thus  \eqref{eq:support} implies that  $E^{m,s}_1$ is non-zero for  $m+s=0$ if and only if $m=s=0$.
We claim  that $E^{0,0}_1$ is one-dimensional, which will complete the proof since the convergence property of Theorem~\ref{main2} then implies that $H^{2d(n-k)}_c( M_{\overline K},\Ql)$ is one-dimensional.  But  $\Pconf_0$ is simply a point, so $H^0_c(\Pconf_0,\Ql)$ is one-dimensional, while  the $0$-fold tensor product of $H^{n-1}_c(X_{\overline K},\Ql)$ is also one-dimensional, as is the sign character. Because $S_0$ is the trivial group, taking $S_0$-invariants changes nothing, which thereby establishes the claim.
\end{proof}

\section{The geometric major arcs}\label{s:major}

Rather than $A_m^{kd}$ we will need to work with a subset of the major arcs which are very close to a rational with denominator having degree precisely $m$. (The reason for this is that we have chosen our parameters so that the exponential sum vanishes outside this set, and we will see in Lemma~\ref{infinityvanishing} that the same is true of $S_{d,f}$.)
  Thus,  let $U_m^{kd}\subset A_m^{kd}$ consist of tuples  $$(b_1,\dots,b_{kd})\in  A_m^{kd} - A_{m-1}^{kd}$$ such that there exists a polynomial $h_1(T)$ of degree $<m$ and a monic polynomial $h_2(T)$ of degree $m$ for which
$$
\sum_{r =1}^{kd} b_r T^{-r}  = \frac{h_1(T)}{h_2(T)} + O (T^{-kd-1 }).
$$
The following result is concerned with a description of this set.

\begin{lemma}\label{rationalisation} Assume that $m \leq kd/2$.
\begin{enumerate} 

\item $U_m^{kd}$ is a Zariski open subset of $A_m^{kd}$. 

\item The coefficients of the polynomials $h_1,h_2$ are regular functions on $U_m^{kd}$, and these give an isomorphism between $U_m^{kd}$ and the space of pairs of  coprime 
polynomials $h_1,h_2\in \FF_q[T]$, such that $h_2$ is monic and 
$\deg(h_1)<m=\deg(h_2)$.

\end{enumerate}
  
  \end{lemma}

  \begin{proof}
For  part (1) we apply the uniqueness statement in Lemma~\ref{majorarcproperties}(5). This implies that a point $(b_1,\dots,b_{kd})$ is 
not in $U_m^{kd}$ if and only if it satisfies the conditions of  Lemma~\ref{majorarcproperties}(4) for some $m' <m$. 
This is so if and only if  $(b_1,\dots,b_{kd-1})$ is in $A^{kd-1}_{m-1} \subseteq \mathbb A^{kd-1}$, 
since  $b_{kd}T^{-kd}=O(T^{-kd-1+m-m'})$ if $m'<m$. 
But this  is a Zariski closed condition by Lemma 
\ref{majorarcproperties}(1).
 
 We now turn to the proof of part (2). 
 The existence of a polynomial map from the space of 
 pairs of 
relatively prime 
polynomials $h_1,h_2$, such that $h_2$ is monic and 
$\deg(h_1)<m=\deg(h_2)$, to 
$U_m^{kd}$ follows immediately from the formula for polynomial long division.
  
The inverse map is not hard to construct. To do so, first observe that
\begin{equation}\label{eq:kiwi}  h_2(T) \sum_{r=1}^{kd} b_r T^{-r} = h_1(T) + O(T^{m-kd-1}).
\end{equation}
Thus  all coefficients of $h_2(T) \sum_{r=1}^{kd} b_r T^{-r}$,  
between $T^{-1}$ and $T^{-m}$ vanish.
This gives $m$ linear equations in the $m$ coefficients of $h_2(T)$ (not counting the leading one, which is fixed). In turn this allows us to write $h_2(T)$ as a polynomial function on the open set where the determinant of the system of equations is non-vanishing. 
We need to prove that the determinant of this system is non-zero
on $U_m^{kd}$.
Suppose for a contradiction that there is a vector in the kernel that defines a polynomial $h_2'$ of degree $<m$ such that all coefficients of $h_2'(T) \sum_{r=1}^{kd} b_r T^{-r} $ between $T^{-1}$ and $T^{-m}$ vanish. Thus  there exists a polynomial $h_1'$ such that \[ h_2'(T) \sum_{r=1}^{kd} b_r T^{-r} = h_1'(T)+ O(T^{-m-1}).\] Multiplying both sides by $h_2$, we obtain 
\begin{align*}
h_2(T) h_1'(T)+ O(T^{-1}) &= h_2'(T) h_2(T) \sum_{r=1}^{kd} b_r T^{-r} \\
&= h_2'(T) (h_1(T) + O(T^{m-kd-1}))\\
&=h_2'(T) h_1(T) + O(T^{2m-kd-2}).
\end{align*} 
It follows that  the polynomials $h_2(T) h_1'(T)$ and $h_2'(T)h_1(T)$ are equal up to $O(T^{-1})$, whence equal. Thus  \[ \sum_{r =1}^{kd} b_r T^{-r}  = \frac{h_1'(T)}{h_2'(T)} + O (T^{-kd-1 }),\] 
which contradicts the asumption that $(b_1,\dots,b_{kd})\not\in A_{m-1}^{kd}$.
Finally, the equation 
\eqref{eq:kiwi}
allows  us to write the coefficients of $h_1$ as polynomial functions in the coefficients of $h_2$ and the $b_r$. Hence, by the previous discussion, as polynomial functions on $U_m^{kd}$.  
 \end{proof}

    The following statement on the cohomology of Artin--Schreier sheaves is a variant of well-known facts, and will be convenient for our purposes.

  \begin{lemma} \label{phasecancellation}
Let $X$ be a variety over a separably closed  field of characteristic $p$, equipped with a map $\varphi: X \to \mathbb A^1$. Let $a: X \times \mathbb A^1 \to X$ be another map such that  $$
\varphi(a(x,\lambda))= \varphi(x)+ u(x) \lambda
$$ 
for an invertible function $u(x)$, and such that $a$ is an action of the group $\mathbb G_a \cong \mathbb A^1$ on $X$. (i.e. $a(x,0)=x$ and $a(a(x,\lambda_1),\lambda_2)=a(x,\lambda_1+\lambda_2)$.) 
 Then  it follows that $H^i_c(X, \varphi^* \mcL_\psi)=0$ for all $i$.
\end{lemma}
  
 \begin{proof}  First note that 
 \begin{align*}
\varphi(x) + u(x)(\lambda_1+\lambda_2) = \varphi(a(x,\lambda_1+\lambda_2))
&= \varphi(a(a(x,\lambda_1),\lambda_2))\\
&= \varphi(a(x,\lambda_1))+u(a(x,\lambda_1))\lambda_2 \\
&= \varphi(x) + u(x) \lambda_1 +u(a(x,\lambda_1))\lambda_2.
\end{align*} 
Hence $u(a(x,\lambda_1))=u(x)$. 
Next, consider the map  $a': X \times \mathbb A^1 \to X \times \mathbb A^1$,  given by $a'(x,\lambda)=(a(x,\lambda/u(x)),\lambda)$. 
Then $a'$ is invertible because 
\begin{align*}
 a\left( a\left(x,\frac{\lambda}{u(x)}\right), -\frac{ \lambda}{ u(a(x,\lambda/u(x)) } \right) 
 &=  a\left( a\left(x,\frac{\lambda}{u(x)}\right), -\frac{ \lambda}{ u(x) } \right)  
 \\
& = a(x,0)\\ &= x,
\end{align*} 
so that an inverse map is given by $(x,\lambda) \mapsto (a(x,-\lambda/u(x)),\lambda)$.
 
Recall (e.g. from \cite[Tag 03RR]{stacks} and Poincar\'{e} duality) that  
$$
H^i_c(\mathbb A^1,\Ql)
=
\begin{cases}
0 &\text { if $i\neq 2$},\\ 
\Ql &\text { if $i=2$}.
\end{cases}
$$
Appealing to the K\"{u}nneth formula, 
we deduce that 
\begin{align*}
H^i_c(X, \varphi^*\mcL_\psi) = H^{i-2}_c(X \times \mathbb A^1, \varphi^* \mcL_\psi \boxtimes \Ql) 
&= H^{i-2}_c(X \times \mathbb A^1, a'^*(\varphi^* \mcL_\psi \boxtimes \Ql)) \\
&= H^{i-2}_c(X \times \mathbb A^1, \varphi^* \mcL_\psi \boxtimes \mcL_\psi)\\ 
&= 0,
\end{align*}
where we have used the fact that $H^i_c(\mathbb A^1,\mcL_\psi)=0$ for all $i$, as follows from  the case $d=1,n=1$ of \cite[Lemma 8.5(i)]{weili}, for example.
 \end{proof}

 Armed with this result we may now investigate the stalk of the complex 
 $S_{d,f}$ at a typical point on our geometric major arcs. This result will take the place of stationary phase arguments that can be used to bound real oscillatory integrals and $p$-adic exponential sums in the classical circle method. We use it at the infinite place in Lemma~\ref{infinityvanishing}, and then at the finite places in Lemma~\ref{powervanishing}.  Our conditions on $f$ are strong enough that in both cases  the relevant sums actually vanish, and we will prove  a corresponding vanishing statement for cohomology.  The vanishing of the exponential sum at finite places corresponds to the fact that the limit $ \lim_{r \to \infty}p^{-(n-1)r}  \#X (\mathbb Z/p^r\ZZ) $ 
 simplifies to $p^{-(n-1)}\#X(\mathbb F_p)$ if $X$ is smooth.  Since
 the place at $\infty$ of a function field is non-archimedean, our smoothness conditions on $f$ also allow one to deduce a similar vanishing statement for  the relevant exponential sums there. 
 
  \begin{lemma}\label{infinityvanishing} 
  Assume that $d \geq m$. 
 Then $H^i_c( A_m^{kd} - A_{m-1}^{kd}, S_{d,f})= H^i_c( U_m^{kd}, S_{d,f}).$
\end{lemma}
  
  \begin{proof} 
   By excision, it suffices to prove that 
      the stalk of $S_{d,f}$ vanishes at any point 
 $(b_1,\dots,b_{kd})\in A_m^{kd} -(A_{m-1}^{kd}\cup U_{m}^{kd})$.
 By proper base change, this stalk 
 is precisely the compactly supported cohomology of $(\mathbb A^d)^n$ with coefficients in $e_{(b_1,\dots,b_{kd})}^* \mcL_\psi$ where $e_{(b_1,\dots,b_{kd})}$ is $e$ restricted to the point $(b_1,\dots,b_{kd})$
  We will show this stalk vanishes using Lemma~\ref{phasecancellation}.
  
  Let $m',h_1,h_2$ be as in Proposition~\ref{majorarcproperties}(4). By the definition of $U_m^{kd}$, 
  we must have $m'<m$. Let $f_0$ be the leading terms of $f$. Because $(x_1:\dots:x_n)$ is a smooth point of the hypersurface $f_0=0$, 
  we may assume without loss of generality that $\frac{\partial f_0}{\partial x_1}(P)\neq 0$.
  
Let $g_1,\dots,g_n$ be a tuple of polynomials of degree $d$ with leading coefficients exactly $(x_1,\dots , x_n)$ and let $\lambda\in \bar\FF_q$.
Consider 
\[
F(T)\hspace{-0.1cm}
=\hspace{-0.1cm}
f(g_1(T) + \lambda T^{d-m} h_2(T) , g_2(T),\dots,g_n(T)) - f(g_1(T),\dots,g_n(T)).\]
To begin with we note that  $F(T)$ 
  is  divisible by $h_2(T)$, as modulo $h_2(T)$ the two terms cancel. 
Next, we 
use Taylor expansion to deduce that 
$$
F(T)=\lambda \widetilde{g}_1(T) \frac{\partial f}{\partial x_1}(g_1(T),\dots,g_n(T))
+O(T^{2(d+m'-m)+(k-2)d}),
$$
where
$\widetilde{g}_1(T)=T^{d-m}h_2(T) $ has degree exactly $d-m+m'$.
Moreover, 
$$2(d+m'-m)+(k-2)d<d-m+m'+(k-1) d
$$ 
for $m'<m$.
It follows that 
 $\deg (F)= d-m+m'+(k-1) d$, 
 with leading coefficient proportional to $\lambda$,  since 
 $\frac{\partial f_0}{\partial x_1}(P)\neq 0$.

 Since $\sum_{r=1}^{kd} b_r T^{-r} - \frac{h_1(T) }{ h_2(T) } $ is a power series in $T$ of degree  $-kd-1 +m-m'$,  the only contributions to the coefficient of $T^{-1}$ in 
$$
\left( \sum_{r=1}^{kd} b_r T^{-r} - \frac{h_1(T) }{ h_2(T) }\right) 
F(T)
$$
come from the leading terms on both sides. Hence the coefficient of $T^{-1}$ is also a non-zero multiple of $\lambda$. On the other hand, because 
$F(T)$
is a multiple of $h_2(T)$, it follows that 
$\frac{h_1(T) }{ h_2(T) } F(T)$ 
is a polynomial in $T$, and so its coefficient of $T^{-1}$ vanishes. Hence the coefficient of $T^{-1}$ in 
\[\left( \sum_{r=1}^{kd} b_r T^{-r} \right)F(T)
\] 
is a non-zero multiple of $\lambda$. 
  
We may now  apply Lemma~\ref{phasecancellation}. We take $X$ to be $(\mathbb A^d)^n$, $\varphi$ to be $e_{(b_1,\dots,b_{kd})}$ and  $a: (\mathbb A^d)^n \times \mathbb A^1 \to (\mathbb A^d)^n$ to be the map that sends  a tuple of polynomials $g_1(T),\dots,g_n(T)$ of degree $d$, with leading coefficients exactly $(x_1,\dots , x_n)$, and a number $\lambda$, to the tuple  $g_1(T) + \lambda T^{d-m} h_2(T) , g_2(T),\dots,g_n(T)$.
It immediately follows that  the compactly supported cohomology of $e_{(b_1,\dots,b_{kd})}^* \mcL_\psi$ vanishes.
   \end{proof}
  
Recall from Lemma~\ref{rationalisation} that $U_m^{kd}$ is isomorphic to the space of pairs of 
relatively prime 
polynomials $h_1,h_2$, such that $h_2$ is monic and 
$\deg(h_1)<m=\deg(h_2)$.
We can rewrite the map $e$ on 
$(\AA^n)^d\times U_m^{kd}$ as the map that sends 
$$
\left(
\left(a_{0,j}, \dots,a_{d-1,j}\right)_{1\leq j\leq n}
, (h_1,h_2)\right)
$$
to the coefficient of $T^{-1}$ in \[ \frac{h_1(T)}{h_2(T)}  f\left( x_1 T^d +   \sum_{i=0}^{d-1} a_{i,1} T^i, \dots, x_n T^d  +   \sum_{i=0}^{d-1} a_{i,n} T^i \right),\] 
since  $ f( x_1 T^d +   \sum_{i=0}^{d-1} a_{i,1} T^i, \dots, x_n T^d  +   \sum_{i=0}^{d-1} a_{i,n} T^i )$ has degree $<kd$ and so it is sufficient to approximate the first term to within $O(T^{-kd-1})$.  In particular, note that it only depends  on the residue class of  the tuple 
$$ 
\left(x_1 T^d +   \sum_{i=0}^{d-1} a_{i,1} T^i, \dots, x_n T^d  +   \sum_{i=0}^{d-1} a_{i,n} T^i\right)
$$ 
modulo $h_2$. 

Consider the map $\rho: \left(\mathbb A^n\right)^d \times U_m^{kd} \to  \left(\mathbb A^n\right)^m \times U_m^{kd}$ given by taking the residue of  $( x_1 T^d +   \sum_{i=0}^{d-1} a_{i,1} T^i, \dots, x_n T^d  +   \sum_{i=0}^{d-1} a_{i,n} T^i)$ modulo $h_2$, which is a polynomial map by Euclid's algorithm for polynomials. (This  crucially uses  the fact that $h_2$ is monic.)  By the aforementioned residue dependence, we may write $e = \overline{e} \circ \rho$, where 
$\overline{e}: (\AA^n)^m\times U_m^{kd}\to \AA^1$ 
sends 
$$
\left(
\left(a_{0,j}, \dots,a_{m-1,j}\right)_{1\leq j\leq n}
, (h_1,h_2)\right)
$$ to the coefficient of $T^{-1}$ in \[ \frac{h_1(T) }{h_2(T) }  f\left(    \sum_{i=0}^{m-1} a_{i,1} T^i, \dots,     \sum_{i=0}^{m-1} a_{i,n} T^i \right).\]
Similarly, we may write $p_2 = p_2' \circ \rho$, where $p_2'
:\left(\mathbb A^n\right)^m \times U_m^{kd} \to U_m^{kd}$ is the 
natural projection.

Let $\pi_2: U_m^{kd} \to \mathbb A^m$ denote the projection to the space of degree $m$ monic polynomials, which we view as $\mathbb A^m$, that sends $(h_1,h_2)$ to $h_2$. 
We introduce the complex 
 $$\overline{S}_{m,f} = R \pi_{2!}  Rp_{2!}' \overline{e}^*  \mcL_\psi
$$
on $\AA^m$. The following result  is the geometric analogue of   breaking the exponential sum  into residue classes
  on the major arcs. 
  
  \begin{lemma}\label{residuereduction} Assume that $m \leq d$. Then $$H^i_c( U_m^{kd}, S_{d,f}) = H^{i-2n(d-m)}_c(\mathbb A^m, \overline{S}_{m,f}) (-n(d-m)).$$
  \end{lemma}
  
  \begin{proof}  First we will show that \[R \rho_! \Ql = \Ql[2n(d-m)](-n(d-m)).\] To do this, we claim that there exists an isomorphism 
$$
 (\mathbb A^d)^n \times U_{m}^{kd} \cong (\mathbb A^{d-m})^n \times (\mathbb A^m)^n \times U_m^{kd}, 
$$
  whose composition with $\rho$ is the projection onto $ (\mathbb A^m)^n \times U_m^{kd} $.  This isomorphism is defined by the fact that, for a fixed monic polynomial $h_2$ of degree $m$, a polynomial of degree $d\geq m$ with leading coefficient $x_i$ can be written uniquely as a polynomial of degree $<m$ plus $h_2$ times a polynomial of degree $d-m$ with leading coefficient $x_i$. We view the polynomial of degree $<m$ as an element of $\mathbb A^m$ and the polynomial of degree $d-m$ with leading coefficient $x_i$ as an element of $\mathbb A^{d-m}$. The map to $\mathbb A^m \times \mathbb A^{d-m}$ is given by polynomial long division, and hence is a polynomial function of the coefficients, and the inverse map is simply multiplication and addition. The map $\rho$ is given by the residue mod $h_2$, which is exactly the polynomial of degree $<m$, as desired.
  
 By proper base change $R \rho_! \Ql$ is the pullback from a point of the compactly supported cohomology of $(\mathbb A^{d-m})^n = \mathbb A^{n(d-m)}$, which is the pullback of a one-dimensional vector space in degree $2n (d-m)$ with Galois action $\Ql(-n(d-m))$, which is just the  constant complex $ \Ql[2n(d-m)](-n(d-m))$. This establishes the claim.
 
 Next, 
on $U_m^{kd}$ we calculate that 
\begin{align*}
S_{d,f} &= R p_{2!} e^* \mcL_\psi= R p_{2!} \rho^*  \overline{e}^* \mcL_\psi = R p_{2!}' \rho_! \rho^* \overline{e}^* \mcL_\psi ,
\end{align*}
by functoriality.
Appealing to the
projection formula, we deduce that 
\begin{align*}
R p_{2!}' \rho_! \rho^* \overline{e}^* \mcL_\psi
&= Rp_{2!}' \left( \overline{e}^* \mcL_\psi \otimes R \rho_! \Ql\right) &\\&=Rp_{2!}'  \overline{e}^* \mcL_\psi [- 2n(d-m)](-n(d-m)). &
\end{align*}
It follows that 
\begin{align*}
H^i_c( U_m^{kd}, S_{d,f})
&= H^i_c(U_m^{kd}, Rp_{2!}' \overline{e}^* \mcL_\psi [- 2n(d-m)](-n(d-m)))\\
&= H^{i-2n(d-m)}_c(U_m^{kd}, Rp_{2!}' \overline{e}^* \mcL_\psi(-n(d-m))),
\end{align*}
by the Leray spectral sequence with compact supports. But then 
\begin{align*}
H^i_c( U_m^{kd}, S_{d,f})
&= H^{i-2n(d-m)}_c(\mathbb A^m, R \pi_{2!} Rp_{2!}' \overline{e}^* \mcL_\psi)(-n(d-m))\\
&= H^{i-2n(d-m)}_c(\mathbb A^m, \overline{S}_{m,f})(-n(d-m)),
\end{align*}
as required.
\end{proof}

Next, in Lemma~\ref{factorisation}, we will show that the complexes $\overline{S}_{m,f}$ enjoy a factorisation property. This is analogous to the multiplicativity property of the corresponding exponential sum in the classical arithmetic  setting. In calculating a multiplicative function, we can typically reduce to the case of prime powers, which here would correspond to powers of irreducible polynomials. In fact, in our setting we may reduce to powers of polynomials of degree $1$. In Lemma~\ref{powervanishing}, we deal with powers of degree greater than $1$. Thus it remains to calculate the degree $1$ case $\overline{S}_{1,f}$.  This we accomplish in Lemma~\ref{mainterm}.  Building on this, in Lemma~\ref{residuecalculation} we are able to calculate $H^*_c ( \mathbb A^m, \overline{S}_{m,f}) $, which allows us in Corollaries~\ref{degreemcalculation} and~\ref{degreemsmooth} to determine the cohomological contribution of $A^{kd}_m - A^{kd}_{m-1}$.

\begin{lemma}\label{factorisation} Let $m_1,m_2\in \NN$  and let $V$ be the moduli space of pairs of coprime monic  polynomials $l_1, l_2$ 
such that $\deg(l_i)=m_i$ for $i=1,2$. Let $f_1: V \to \mathbb A^{m_1}$
(resp.~$f_2: V \to \mathbb A^{m_2}$,  $f_{12}: V \to \mathbb A^{m_1+m_2}$) be the maps sending $(l_1,l_2)$ to $l_1$ (resp.~$l_2$, $l_1l_2$).  
Then $$f_{12}^* \overline{S}_{m_1+m_2,f} = f_1^* \overline{S}_{m_1,f} \otimes f_2^* \overline{S}_{m_2,f}.$$
If $m_1=m_2$, then we may take this isomorphism to commute with the action of the involution switching $l_1$ and $l_2$ on both sides. 
\end{lemma}

In the $m_1=m_2$ case, the action of the involution on the right hand side follows  the standard convention for switching the two sides of a tensor product of complexes (which acts on the $i$th homology of the first complex tensor the $j$th homology of the second complex by the obvious action times $(-1)^{ij}$). 

\begin{proof}[Proof of Lemma~\ref{factorisation}] 
Applying proper base change, we see that the left hand side is the compactly supported pushforward to $V$ from the space of pairs of $h_1$ a polynomial of degree $<m_1+m_2$, relatively prime to $l_1 l_2$, and a tuple $(g_1,\dots,g_n)$ of polynomials of degree $<m_1+m_2$ of the pullback of $\mcL_\psi$ along the map defined by the coefficient of $T^{-1}$ in \[ \frac{h_1(T)}{l_1(T) l_2(T) }  f\left(  g_1,\dots,g_n \right).\]  

Applying proper base change and the K\"{u}nneth formula, the right hand side is the compactly supported pushforward to $V$ from the space of quadruples of $h_{1,1}$ a polynomial of degree $<m_1$, relatively prime to $l_1$, $h_{1,2}$ a polynomial of degree $<m_2$, 
relatively prime to $l_2$, 
$(g_{1,1},\dots,g_{n,1})$ a tuple of polynomials of degree $<m_1$, and $(g_{1,2},\dots,g_{n,2})$ a tuple of polynomials of degree $<m_2$ of the pullback of $\mcL_\psi$ along the map defined by the coefficient of $T^{-1}$ in \[ \frac{h_{1,1}(T)}{l_1(T) }  f\left(  g_{1,1},\dots,g_{n,1} \right)+\frac{h_{1,2}(T)}{l_2(T) }  f\left(  g_{1,2},\dots,g_{n,2} \right). \]  
We will show that these varieties are actually isomorphic in a way preserving their projections to $V$ and preserving the maps on which we are pulling back $\mcL_\psi$. This is sufficient to imply the isomorphism of compactly supported pushforwards. 

To write a map from the second to the first, let $g_i$ be the unique polynomial of degree $<m_1+m_2$ that is congruent to $g_{i,1}\bmod{l_1}$ and congruent to $g_{i,2}\bmod{l_2}$. This can be written explicitly as $l_2 ( g_{i,1} l_2^{-1} \bmod{ l_1}) + l_1( g_{i,2} l_1^{-1} \bmod{l_2})$, where the inverses are understood to be modulo $l_1$ and $l_2$, respectively. This is a polynomial function on this moduli space of tuples of polynomials. (Here we use the fact that $l_1$ and $l_2$ are relatively prime to make their inverses modulo each other polynomial, together with Euclid's algorithm and the fact that they are monic to make the modulo operation polynomial.)  
 Let $h_1 = h_{1,1}l_2 + h_{1,2}l_1$. Then the following identities hold in the group of formal Laurent series in $T^{-1}$ modulo polynomials in $T$:
\begin{align*}
 \frac{h_{1,1}(T)}{l_1(T) }  &f\left(  g_{1,1},\dots,g_{n,1} \right)+\frac{h_{1,2}(T)}{l_2(T) }  f\left(  g_{1,2},\dots,g_{n,2}\right)\\
 & =  \frac{h_{1,1}(T)}{l_1(T) }  f\left(  g_{1},\dots,g_{n} \right)+\frac{h_{1,2}(T)}{l_2(T) }  f\left(  g_{1},\dots,g_{n} \right) \\
 &= \frac{h_{1,1}(T)l_2(T) + h_{1,2}(T) l_1(T)}{l_1(T) l_2(T)}f\left(  g_{1},\dots,g_{n}\right)\\
 &= \frac{h_1(T)}{l_1(T) l_2(T) }f\left(  g_{1},\dots,g_{n} \right), 
 \end{align*}  
and so the coefficients of $T^{-1}$ on both sides are equal as desired.
Furthermore, it is easy to check that 
 the inverse to this map is the map that sets $g_{i,1}$ to the remainder of $g_i$ modulo $l_1$, $g_{i,2}$ the remainder of $g_i$ modulo $l_2$,  $h_{1,1} = h_1 l_2^{-1} \mod l_1$ and $h_{1,2} = h_1 l_1^{-1} \mod l_2$.
 
Since  this isomorphism between the underlying spaces commutes with the involution switching $h_1$ and $h_2$, the symmetry of the final isomorphism follows from the symmetry of definition of the K\"{u}nneth formula. \end{proof}
 
 We take advantage of the fact that $X$ is smooth to show cancellation in the exponential sums associated to non-squarefree moduli. This is the only place in the proof of Corollary~\ref{degreemcalculation} where the smoothness of $X$ is used. However, it is needed again in Lemma~\ref{supersmooth} to prove Proposition~\ref{majormain}.  
 
 \begin{lemma}\label{powervanishing} Let $h_2$ be a monic polynomial of degree $m$ that is not squarefree. Then the stalk of $\overline{S}_{m,f}$ at $h_2$ vanishes. \end{lemma}
 
 \begin{proof}   By Lemma~\ref{factorisation} it suffices to handle the case where $h_2 = (T-x)^m$ is a power of a linear polynomial.  By proper base change, the stalk of $\overline{S}_{m,f}$ at this point is the cohomology with compact supports of the space of tuples 
$(
(a_{0,j}, \dots,a_{m-1,j})_{1\leq j\leq n}
,h_1)$,
where $h_1$ is a polynomial  of degree $<m$ that is  coprime to $T-x$, of the cohomology of $\overline{e}^* \mathcal L_\psi$, where we recall that $\overline{e}$ sends a tuple to the coefficient of $T^{-1}$ in \[ \frac{h_1(T) }{h_2(T) }  f\left(    \sum_{i=0}^{m-1} a_{i,1} T^i, \dots,      \sum_{i=0}^{m-1} a_{i,n} T^i \right).\]
By a further proper base change, we can consider the map $r$ from this space to $\mathbb A^n$ defined by the coordinates $(\sum_{i=0}^{m-1} a_{i,1} x^i,\dots, \sum_{i=0}^{m-1}a_{i,n} x^i)$. It is sufficient to show that the stalk of the compactly supported pushforward  
$Rr_! \overline{e}^* \mathcal L_\psi$ vanishes everywhere.  This stalk is the cohomology with compact supports of the same variety with coefficients in the same sheaf, but with the tuple  $(\sum_{i=0}^{m-1} a_{i,1} x^i,\dots, \sum_{i=0}^{m-1}a_{i,n} x^i)$ restricted to fixed values. 
 
 We split into two cases, according to whether or not  
 $$
 f\left(\sum_{i=0}^{m-1} a_{i,1} x^i,\dots, \sum_{i=0}^{m-1}a_{i,n} x^i\right)
 $$ 
is zero.
Suppose first that it is zero.
Then it follows from the smoothness of $X$ that  
$\frac{\partial f}{\partial x_{j}}(\sum_{i=0}^{m-1} a_{i,1} x^i,\dots, \sum_{i=0}^{m-1}a_{i,n} x^i)\neq 0$
for some $j$. 
Consider the automorphism 
$$
a_{i,j}\mapsto 
a_{i,j} +  
\lambda{m-1 \choose i} (-x)^{m-1-i} .
$$
It acts on $\sum_{i=0}^{m-1} a_{i,j} T^i$ by adding $\lambda (T-x)^{m-1}$.
Thus it acts on $$f\left(    \sum_{i=0}^{m-1} a_{i,1} T^i, \dots,      \sum_{i=0}^{m-1} a_{i,n} T^i \right) 
$$
by adding $\lambda  (T-x)^{m-1}\frac{\partial f}{\partial x_{j} }(\sum_{i=0}^{m-1} a_{i,1} x^i,\dots, \sum_{i=0}^{m-1}a_{i,n} x^i)$ plus higher powers of $T-x$. In particular,  it acts on \[\frac{h_1(T) }{h_2(T) }  f\left(    \sum_{i=0}^{m-1} a_{i,1} T^i, \dots,      \sum_{i=0}^{m-1} a_{i,n} T^i \right)\] by adding 
\[\lambda \frac{h_1(T)}{T-x} \frac{\partial f}{\partial x_{j}} \left(\sum_{i=0}^{m-1} a_{i,1} x^i,\dots, \sum_{i=0}^{m-1}a_{i,n} x^i\right) 
\] 
plus a power series in $T$. Thus it acts on $\overline{e} (
(a_{0,j}, \dots,a_{m-1,j})_{1\leq j\leq n}
,h_1)$
by adding $\lambda h_1(x) \frac{\partial f}{\partial x_{j}}$. Because $h_1$ is prime to $T-x$, this is a non-zero multiple of $\lambda$, and so the compactly supported cohomology vanishes by Lemma~\ref{phasecancellation}.
 
 If, on the other hand,  $f(\sum_{i=0}^{m-1} a_{i,1} x^i,\dots, \sum_{i=0}^{m-1}a_{i,n} x^i)\neq 0$, then we consider the automorphism  $h_1(T)\mapsto h_1(T) + \lambda (T-x)^{m-1}$.  
 Recalling that $h_2(T)=(T-x)^m$,
 this automorphism acts on  \[ \frac{h_1(T) }{h_2(T) }  f\left(    \sum_{i=0}^{m-1} a_{i,1} T^i, \dots,   
 \sum_{i=0}^{m-1} a_{i,n} T^i\right)\] by adding 
 \begin{align*}
\lambda \frac{1}{T-x}  f\left(    \sum_{i=0}^{m-1} a_{i,1} T^i, \dots,   \sum_{i=0}^{m-1} a_{i,n} T^i\right)
&=  \lambda f\left(    \sum_{i=0}^{m-1} a_{i,1} x^i, \dots,      \sum_{i=0}^{m-1} a_{i,n} x^i\right),
\end{align*}
plus a power series in $T$. Hence the coefficient of $T^{-1}$ is a non-zero multiple of $\lambda$ and the compactly supported cohomology vanishes by Lemma~\ref{phasecancellation}.
\end{proof}

 Let $H^{*,\red}_c(X_{\bar \FF_q},\Ql)$ be the mapping cocone of the trace map from $H^*_c(X_{\bar \FF_q},\Ql) $ to $ \Ql[-2(n-1)](-(n-1))$, which exists because $X$ is a variety of dimension $n-1$. In the next two results, we will show that $H^{*,\red}_c(X_{\bar \FF_q},\Ql)$ is the solution to a sheaf cohomology problem that turns out to be precisely what is needed to calculate $\overline{S}_{1,f}$

 \begin{lemma}\label{redcoh}Assume that $X$ is irreducible and that $\chaar(\mathbb F_{q})>k$. Then the compactly supported cohomology of $\mathbb G_m \times \mathbb A^n$ with coordinates $h,a_1,\dots,a_n$, with coefficients in $\mcL_\psi ( h f(a_1,\dots,a_n))$, is $H^{*,\red}_c(X_{\bar \FF_q},\Ql) [-2](-1)$. \end{lemma}
 
 \begin{proof}
By excision, there is a long exact triple 
\begin{align*}
H^*_c ( \mathbb G_m \times \mathbb A^n,\mcL_\psi ( h f(a_1,\dots,a_n)))
&\to H^*_c ( \mathbb A^1 \times \mathbb A^n,\mcL_\psi ( h f(a_1,\dots,a_n)))\\ 
&\to H^*_c (  \mathbb A^n,\mcL_\psi ( 0 f(a_1,\dots,a_n))),
\end{align*}
where $\mcL_\psi ( 0 f(a_1,\dots,a_n))$ is the constant sheaf obtained by pullback 
to the hyperplane $h=0$ 
of the sheaf 
$\mcL_\psi ( h f(a_1,\dots,a_n))$.
The middle complex is  $H^*_c(X_{\bar \FF_q},\Ql) [-2](-1)$ and the 
 third complex is equal to $\Ql[-2n](-n)$. It remains to check that this map is  a non-zero multiple of the trace map, which we do by 
 bounding the degrees in which
 $H^*_c ( \mathbb G_m \times \mathbb A^n,\mcL_\psi ( h f(a_1,\dots,a_n)))$ is non-vanishing. 

Assume without loss of generality that $f$ actually depends on $a_1$. We will show that for generic $a_2,\dots,a_{n-1}$, the compactly supported cohomology of $\mathbb G_m \times \AA^1$ 
with coefficients in $\mcL_\psi(h f(a_1,a_2\dots,a_n))$ vanishes in degree greater than $2$, and for arbitrary $a_2,\dots,a_{n-1}$ it vanishes in degree greater than $4$.

The second statement is simply a consequence of the cohomological dimension of $\mathbb G_m \times \mathbb A^1$. 

For the first statement, we have some polynomial $f(a_1,a_2,\dots,a_n)$ of $a_1$, of degree $<p$, which because $a_2,\dots,a_n$ are generic is non-constant. Hence, on taking the cohomology along $\mathbb A^1$ of the corresponding Artin--Schreier sheaf, we obtain a complex supported in degree $1$. Furthermore, we can represent this complex as the Fourier transform of $f_* \mathbb Q_\ell$, shifted by $1$. Because $f_* \mathbb Q_\ell$ is a middle extension sheaf, its Fourier transform does not have a constant sheaf as a quotient by \cite[8.2.5(2)]{GKSM}, and hence its compactly supported cohomology is supported in degree $2$, as desired. 

It follows from these calculations, and by using  the cohomological dimension of $\mathbb A^{n-1}$, that $H^*_c ( \mathbb G_m \times \mathbb A^n,\mcL_\psi ( h f(a_1,\dots,a_n)))$ is supported in degrees $\leq 2(n-1) + 2=2n$ or $\leq 2(n-2)+4=2n$.
Hence the map in degree $2n$ from $H^*_c(X_{\bar \FF_q},\Ql) [-2](-1)$ to $\Ql[-2n](-n)$ is surjective, whence an isomorphism. Thus it must be a non-zero multiple of the trace map.
\end{proof}
 
\begin{lemma}\label{mainterm}  Assume that $X$ is irreducible and that $\chaar(\mathbb F_{q})>k$. Then $\overline{S}_{1,f}$ on $\mathbb A^1$ is the constant complex $H^{*,\red}_c(X_{\bar \FF_q},\Ql) [-2](-1)$. \end{lemma}

\begin{proof} First we show that $\bar{S}_{1,f}$ is a constant complex. In the degree $1$ case, the possible values of $h_2$ are simply $T-x$ for arbitrary $x$, and the possible values of $h_1$ are non-zero constants, and $a_{i,j}$ is simply an $n$-tuple of numbers $a_i$.  Thus the relevant space has coordinates $x, h ,a_1,\dots,a_n$, with $h \neq 0$, where $\overline{e}$ can be written as the coefficient of $T^{-1}$ in $\frac{h}{T-x}  f(a_1,\dots,a_n)$, and $\pi_2 \circ p_2'$ is simply the map $x$.

The coefficient of $T^{-1}$ in  $\frac{h}{T-x}  f(a_1,\dots,a_n)$ is equal to $h f(a_1,\dots,a_n)$. Because this map is independent of $x$, proper base change implies that  the complex $\bar{S}_{1,f} = R (\pi_2 \circ p_2')_!  \overline{e}^* \mcL_\psi$ is constant; viz.\ it is the pullback from a point to $\mathbb A^1$ of the compactly supported cohomology of $\mathbb G_m \times \mathbb A^n$ with coordinates $h,a_1,\dots,a_n$, with coefficients in $\mcL_\psi ( h f(a_1,\dots,a_n))$. Hence, by Lemma~\ref{redcoh} it is equal to the constant complex $H^{*,\red}_c(X_{\bar \FF_q},\Ql) [-2](-1)$.  \end{proof}

We have now completed the calculation of $\overline{S}_{1,f}$, and we are ready to turn around and apply our previous results, using $\overline{S}_{1,f}$ to calculate $\overline{S}_{m,f}$, from there to calculate the cohomology of $\overline{S}_{m,f}$, from there to calculate the cohomology of $S_{d,f}$ on $A^{kd}_m - A^{kd}_{m-1}$, and finally from there to calculate the cohomology of $S_{d,f}$ on $A^{kd}_m$.

\begin{lemma}\label{residuecalculation}Assume that $X$ is irreducible and that $\chaar(\mathbb F_{q})>k$. Then  \[H^*_c ( \mathbb A^m, \overline{S}_{m,f}) = (H^*_c(\Pconf_m,\Ql) \otimes H^{*,\red}_c(X_{\bar \FF_q},\Ql)^{\otimes m})^{S_m} [-2m](-m) .\] \end{lemma}

\begin{proof} Lemma~\ref{powervanishing} yields
$H^*_c ( \mathbb A^m, \overline{S}_{m,f}) = H^*_c ( \Conf_m, \overline{S}_{m,f})$, 
where $\Conf_m$ is the space of squarefree polynomials of degree $m$. Let 
$
r: \Pconf_m \to \Conf_m$
be the natural covering map that sends $(x_1,\dots,x_m)$ to $\prod_{i=1}^m (T-x_i)$. Then $r$ is a Galois finite \'{e}tale $S_m$-cover, so that
\[ H^*_c ( \Conf_m, \overline{S}_{m,f})= H^*_c ( \Pconf_m, r^*\overline{S}_{m,f})^{S_m}. 
\]
By iteratively applying Lemma~\ref{factorisation}, we see that $r^*\overline{S}_{m,f}= \bigotimes_{i=1}^m pr_i^* \overline{S}_{1,f}$ where $pr_i : \Pconf_m \to \mathbb A^1$ is the map sending $(x_1,\dots,x_m)$ to $x_i$,
with $S_m$ acting on it the usual way. Applying Lemma~\ref{mainterm}, this is the constant complex $H^{*,\red}_c(X_{\bar \FF_q},\Ql)^{\otimes m} [-2m](-m) $, with $S_m$ acting on it in the usual way. Hence by the formula for the cohomology of constant complexes
\begin{align*}
H^*_c ( \Conf_m, \overline{S}_{m,f})&=
H^*_c ( \Pconf_m, r^*\overline{S}_{m,f})^{S_m} \\
&= H^*_c ( \Pconf_m, H^{*,\red}_c(X_{\bar \FF_q},\Ql)^{\otimes m} [-2m](-m) )^{S_m}\\
&= \left( H^*_c(\Pconf_m,\Ql) \otimes H^{*,\red}_c(X_{\bar \FF_q},\Ql)^{\otimes m} [-2m](-m)\right)^{S_m}, 
\end{align*}
as required. 
\end{proof}

\begin{cor}\label{degreemcalculation}
Assume that $X$ is irreducible and that $\chaar(\mathbb F_{q})>k$. Then the cohomology group
$H^*_c( A_m^{kd} - A_{m-1}^{kd}, S_{d,f})$ is equal to 
\[  \left( H^*_c(\Pconf_m,\Ql) \otimes H^{*,\red}_c(X_{\bar \FF_q},\Ql)^{\otimes m}  \right)^{S_m}[-2m - 2n(d-m))](-m-n(d-m)). 
\]  
\end{cor}

\begin{proof}
It follows from Lemmas~\ref{infinityvanishing} and 
\ref{residuereduction} that 
\begin{align*}
H^*_c( A_m^{kd} - A_{m-1}^{kd}, S_{d,f}) 
=H^*_c( U_m^{kd}, S_{d,f}) 
&=H^*_c(\mathbb A^m,\overline{S}_{m,f})[-2n(d-m)](-n(d-m)).
\end{align*}
The corollary now follows from an application of   Lemma~\ref{residuecalculation}.\end{proof}

\begin{lemma}\label{supersmooth} Assume that the leading terms of $f$ define a smooth hypersurface and that $\chaar(\mathbb F_{q})>k$. 
Then $H^{*,\red}_c(X
_{\overline\FF_q} 
,\Ql)$ is supported in degree $n-1$ and equals $H^{n-1}_c(X
_{\overline\FF_q} 
,\Ql)$ in that degree. \end{lemma}

\begin{proof} The fact that $H^i_c(X
_{\overline\FF_q} 
,\Ql) = H^{i,\red}_c(X
_{\overline\FF_q} 
,\Ql)$ for $i < 2(n-1)$ follows immediately from the definition. The fact that $H^{2(n-1),\red}(X
_{\overline\FF_q} 
,\Ql)=0$ follows from the definition and the fact that $X$ is irreducible, which means that the trace map is an isomorphism in top degree. It remains to show that $H^i_c(X
_{\overline\FF_q} 
,\Ql)=0$ for $i<2(n-1)$ with $i\neq n-1$. 

For $i<n-1$ this follows immediately from Poincar\'e duality and the fact that the cohomological dimension of an $(n-1)$-dimensional affine variety is $n-1$ by \cite[Expos\'e XIV, Cor.~3.2]{sga4-3}.

For $i>n-1$, we let $\overline{X}$ be the projective closure of $X$  and we let  $D$ denote  the divisor at $\infty$. Then $\overline{X}$ and $D$ are both smooth projective hypersurfaces, with $\dim \overline{X}=n-1$  and  $\dim D=n-2$.
We claim that the restriction map $H^i(\overline{X}
_{\overline\FF_q} 
,\Ql) \to H^i(D
_{\overline\FF_q} 
,\Ql)$ is an isomorphism in every even degree $i$ satisfying $n-1<i < 2(n-1)$ and that it is surjective in degree $n-1$. 
To see this we note that in every degree $i<2(n-1)$, with $i\neq n-1$, 
 the cohomology group $H^i(X
 _{\overline\FF_q} 
 ,\Ql)$ is the one-dimensional space generated by the $\frac{i}{2}$th power of the hyperplane class, if $i$ is even, or vanishes if $i$ is odd.
The same is true for $H^i(D_{\overline\FF_q} 
,\Ql)$ in degrees $< 2(n-1)$ except for $n-2$. Because the pullback of the hyperplane class is the hyperplane class, the pullback map is an isomorphism for $i>n-1$ and surjective for $i=n-1$.
We can apply this fact to the excision exact sequence $$H^i_c(X_{\overline\FF_q} 
,\Ql) \to H^i(\overline{X}_{\overline\FF_q} 
,\Ql) \to H^i(D_{\overline\FF_q} 
,\Ql),
$$ 
whence  $H^i_c(X_{\overline\FF_q} 
,\Ql)$ vanishes in every degree $i\in (n-1, 2(n-1))$. 
\end{proof}

\begin{cor}\label{degreemsmooth} 
Assume that the leading terms of $f$ define a smooth hypersurface, 
 that $\chaar(\mathbb F_{q})>k$ and that $d\geq m$.  Then  
$H^i_c( A_m^{kd} - A_{m-1}^{kd}, S_{d,f})$ is isomorphic to 
\[ \left(H^{a}_c(\Pconf_m,\Ql) \otimes H^{n-1}_c(X_{\overline\FF_q} 
,\Ql)^{\otimes m} \otimes \sgn^{n-1} \right)^{S_m} (-m-n(d-m)), \] 
where $a=i+m(n-1) -2nd$ and  $\sgn$ is the sign representation of $S_m$.
\end{cor} 

\begin{proof}   Corollary~\ref{degreemcalculation} implies that 
$H^*_c( A_m^{kd} - A_{m-1}^{kd}, S_{d,f})$ is isomorphic to 
\[   \left( H^*_c(\Pconf_m,\Ql) \otimes H^{*,\red}_c(X_{\overline\FF_q} 
,\Ql)^{\otimes m}  \right)^{S_m}[-2m - 2n(d-m))](-m-n(d-m)).
\]  
By Lemma~\ref{supersmooth}, $H^{*,\red}_c(X_{\overline\FF_q} 
,\Ql)$ is supported in degree $n-1$ and equal to $H^{n-1}_c(X_{\overline\FF_q} 
,\Ql)$ in that degree. Thus $H^{*,\red}_c(X_{\overline\FF_q} 
,\Ql)^{\otimes m} $ is supported in degree $m(n-1)$ and is equal to $H^{n-1}_c(X_{\overline\FF_q} 
,\Ql)^{\otimes m}$ in that degree. However, if $n-1$ is odd then the $S_m$-action on this tensor power is not the usual one but is instead the usual one twisted by the sign character (because the symmetry of the tensor product of the odd degree cohomology groups of two complexes is the opposite of the usual symmetry).   In this way we deduce that 
 $H^i_c( A_m^{kd} - A_{m-1}^{kd}, S_{d,f}) $ is as claimed,
with  $a=i-2m -2n(d-m) -m (n-1)=i+ m(n-1) -2nd $. 
\end{proof} 

We can make the isomorphism in this result explicit by unwinding all the proofs. This is illustrated 
in Figure \ref{fig:zig}, in which we have omitted the Tate twists for compactness of 
notation.
\begin{figure}
\[ 
\begin{tikzcd}
[column sep=tiny] 
H^i_c( A_m^{kd} - A_{m-1}^{kd}, S_{d,f}) &  H^i_c( U_m^{kd}, S_{d,f} ) \arrow[l,"(1)"] 
\\
H^{i - 2n(d-m)}_c (  \mathbb A^m, R\pi_{2!} R p'_{2!} \overline{e}^* \mathcal L_\psi )  
\arrow[r,"(3)"]
  &  H^{i - 2n(d-m)}_c ( U_m^{kd}, R\pi_{2!} R p'_{2!} \overline{e}^* \mathcal L_\psi )\arrow[u,"(2)"]  \\
H^{i-2n(d-m)}_c ( \mathbb A^m, \overline{S}_{m,f} ) \arrow[u,"(4)"]
&
H^{i-2n(d-m)}_c ( \Conf_m, \overline{S}_{m,f} )  \arrow[l,"(5)"] \\
H^{i-2m-2n(d-m)}_c ( \Pconf_m, H^{*, \operatorname{red}}_c (X_{\overline{\mathbb F}_q}, \mathbb Q_\ell)^{\otimes m} )\arrow[r,"(7)"] &
H^{i-2n(d-m)}_c ( \Pconf_m, r^* \overline{S}_{m,f} )\arrow[u,"(6)"] \\
H^{a }_c ( \Pconf_m,  \mathbb Q_\ell) \otimes H^{n-1}_c(X_{\overline{\mathbb F}_q},\mathbb Q_\ell)^{\otimes m} \otimes \sgn^{n-1} \arrow[u,"(8)"]  &  \\
\left( H^{a}_c ( \Pconf_m,  \mathbb Q_\ell) \otimes H^{n-1}_c(X_{\overline{\mathbb F}_q},\mathbb Q_\ell)^{\otimes m} \otimes \sgn^{n-1} \right)^{S_m} \arrow[u,"(9)"] 
\end{tikzcd}
\]
\caption{}
\label{fig:zig}
\end{figure}
The maps in this composition arise in the following way:
\begin{enumerate}
\item arises from functoriality of compactly supported cohomology under open immersions (and is checked in Lemma~\ref{infinityvanishing} to be an isomorphism);
\item arises from an isomorphism of complexes on $U_m^{kd}$ constructed in Lemma~\ref{residuereduction};
\item is the projection formula along $\pi_2$ (and thus is automatically an isomorphism);
\item is the definition of $\overline{S}_{m,f}$ (and thus is automatically an isomorphism);
\item arises from functoriality of compactly supported cohomology under open immersions (and is checked in Lemma \ref{powervanishing} to be an isomorphism);
\item arises from an isomorphism of complexes constructed in Lemma \ref{residuecalculation};
\item is the trace map of compactly supported cohomology under finite morphisms (which is an isomorphism on $S_m$-invariants because the finite morphism is a Galois finite \'{e}tale cover with Galois group $S_m$); 
\item arises from the isomorphism $  H^{n-1}_c(X_{\overline{\mathbb F}_q},\mathbb Q_\ell)[1-n] \to H^{*, \operatorname{red}}_c (X_{\overline{\mathbb F}_q}, \mathbb Q_\ell) $ and $a=i-2m-2n(d-m) -m (n-1) $; and 
\item is the inclusion of $S_m$ invariants (and so is  an isomorphism on $S_m$-invariants).
\end{enumerate}

We now have all the ingredients to complete our  first stage in the treatment of the geometric major arcs, as enshrined in 
 Proposition~\ref{majormain}. It follows from the construction of the (descending) filtration spectral sequence (e.g  \cite[012K]{stacks}) that it converges to $H^{*}_c(A_d^{kd}, S_{d,f})$, that its first page $E^{m,s}_1$ is the $(m+s)$th cohomology of the $m$th associated graded piece 
of this filtration, which by excision is 
$$H^{m+s}_c( (A_d^{kd} - A_{m-1}^{kd} ) - (A_d^{kd} - A_m^{kd}),S_{d,f}) = H^{m+s}_c(A_m^{kd} - A_{m-1}^{kd},S_{d,f}),
$$ 
and finally that the differential on the $r$th page has bidegree $(r,1-r)$.
 Hence  Corollary~\ref{degreemsmooth} implies that it is 
\[  \left(H^{s+mn-2nd}_c(\Pconf_m,\Ql) \otimes H^{n-1}_c(
X_{\overline\FF_q} 
,\Ql)^{\otimes m} \otimes \sgn^{n-1} \right)^{S_m} (-m-n(d-m)),\] 
since  $m+s+m(n-1) -2nd=s+mn-2nd$. 
This therefore completes the proof of  everything in
   Proposition~\ref{majormain} apart from the vanishing of the differentials. 
  This is the topic of the following section. 
   
     \section{Vanishing of differentials on odd pages} \label{s:diff}
 
The primary goal of this section is to establish the following result, which serves to conclude the proof of 
Proposition~\ref{majormain}.

   \begin{prop}\label{odd-vanishing} 
   In Proposition \ref{majormain}, the spectral sequence has differentials vanishing on every odd page. \end{prop}

  Let $\mathcal M_{n,k}$ be the moduli space over ${\mathbb F}_q$ of pairs of a degree $k$ polynomial $f$ in $n$ variables and a point $(x_1,\dots, x_n) \in \mathbb A^n \setminus\{0\}$, such that the vanishing set of $f$ is smooth, the projective vanishing set of the degree $k$ part of $f$ is smooth, and the projective vanishing set of the degree $k$ part of $f$ contains 
  $(x_1: \dots : x_n)$. Each point of $\mathcal M_{n,k}$ defines a space $\mor_{d,P}(\AA^1,X)$ with $X$ the affine vanishing set of $f$ and $P = (x_1: \dots : x_n)$, and moreover, there is a universal family over $\mathcal M_{n,k}$ whose fiber is $\mor_{d,P}(\AA^1,X)$. Thus $\mathcal M_{n,k}$ is a natural setting to study how the cohomology groups discussed in the rest of the paper vary as $f$ and $P$ do.
  
  Let $\mathcal H_{n,k}$ be the moduli space of smooth projective hypersurfaces of degree $k$ in $\mathbb P^n$. On $\mathcal H_{n,k}$, let $\operatorname{Prim}_{n,k} $ be the lisse sheaf consisting of the primitive middle cohomology of the corresponding family of hypersurfaces. (This is the $(n-1)$th higher pushforward along the universal family of the constant sheaf if $n$ is even, or its quotient by the $\frac{1}{2}(n-1)$th power of the hyperplane class if $n$ is odd.)
    Let $h_n:\mathcal M_{n,k} \to \mathcal H_{n,k}$ be the map given by homogenizing the polynomial $f$, and let $h_{n-1}:\mathcal M_{n,k} \to \mathcal H_{n-1,k}$ be the map given by taking the leading 
terms of $f$. 
  
  We can repeat most of the constructions of  \S~\ref{s:major} 
  in the relative setting. In particular we can define $S_{d,f}$ as a complex of sheaves on $\mathcal M_{n,k} \times \mathbb A^{kd}$. Let $pr_{\mathcal M,m}$ be the projection from $ \mathcal M_{n,k} \times  (A^{kd}_m - A^{kd}_{m-1})$ to $\mathcal M_{n,k}$, where
  $A_m^{kd}$ is  the set of  major arcs in Definition \ref{def:major}.
  
  Let $X_{\mathcal M_{n,k}}$ be the universal family of affine hypersurfaces in $\mathbb A^n$ over $\mathcal M_{n,k}$, let $\overline{X}_{\mathcal M_{n,k}}$ be its projective closure, the universal family of smooth projective hypersurfaces defined by the homogenization of $f$, and let $D_{\mathcal M_{n,k}} $ 
  be the divisor at $\infty$, which is the universal family of smooth projective hypersurfaces defined by the degree $k$ part of $f$.
  
      Let $pr_{\mathcal M,X}$ be the projection from $X_{\mathcal M_{n,k}}$ to $\mathcal M_{n,k}$.  
  We begin by recording some technical  facts
about  the complex $Rpr_{\mathcal M,X!}  \Ql$. 
  
  \begin{lemma}\label{monodromy-of-hypersurfaces} 
  \begin{enumerate}
  \item
    The complex $Rpr_{\mathcal M,X!}  \Ql$ is supported in degrees $n-1$ and $2(n-1)$ and has lisse cohomology in those degrees.
    \item
  $R^{n-1} pr_{\mathcal M,X!}  \Ql$ is an extension of $h_{n}^* \operatorname{Prim}_{n,k}$ by $h_{n-1}^* \operatorname{Prim}_{n-1,k}$.

\item
$R^{2(n-1)} pr_{\mathcal M,X!}  \Ql$ is constant of rank one.  
 
  \end{enumerate}
\end{lemma}

  \begin{proof} The fact that $Rpr_{\mathcal M,X!}  \Ql$ is supported in degrees $n-1$ and $2(n-1)$ follows, upon taking stalks, from Lemma \ref{supersmooth}. To check that the cohomology sheaves are lisse, and calculate them, we introduce some additional notation.  Let $pr_{\mathcal M, \overline{X}}$ be the projection from $\overline{X}_{\mathcal M_{n,k}}$ to $\mathcal M_{n,k}$ and let $pr_{\mathcal M,D}$ be the projection from $D_{\mathcal M_{n,k}}$ to $\mathcal M_{n,k}$. Then we have a distinguished triangle
  \[ R pr_{\mathcal M, X !} \Ql \to R pr_{\mathcal M, \overline{X} !} \Ql \to R pr_{\mathcal M,D!}  \Ql .\]
 Since $pr_{\mathcal M , \overline{X}}$ and $pr_{\mathcal M, D}$ are smooth and projective, with the smoothness following from the definition of $\mathcal M_{d,n}$, it follows that both $R pr_{\mathcal M, \overline{X} !} \Ql $ and $ R pr_{\mathcal M,D !} \Ql$ have lisse cohomology sheaves. This implies part (1).
   
  In degree $n-1$, we have an exact sequence
  \[ 
     \begin{tikzcd}
 R^{n-2}  pr_{\mathcal M, \overline{X} ! }\Ql 
\arrow[r,"i^*" ]
&  R^{n-2}  pr_{\mathcal M,D!} \Ql 
\arrow[r]
&   R^{n-1}  pr_{\mathcal M, X !} \Ql 
  \arrow[d]\\
      &
   R^{n-1}  pr_{\mathcal M, D!} \Ql 
&    \arrow[l, "i^*"] R^{n-1}  pr_{\mathcal M,\overline{X}!} \Ql 
\end{tikzcd}
\] where we have labeled the pullback map $i^*$ along the inclusion $i: D \to \overline{X}$. 
For even  $n$, this exact sequence specialises to 
  \[ 
     \Ql \left(- \tfrac{n-2}{2} \right)
\to 
  \Ql \left(- \tfrac{n-2}{2} \right)  \oplus h_{n-1}^* \operatorname{Prim}_{n-1,k} 
\to 
   R^{n-1}  pr_{\mathcal M, X !} \Ql 
\to  h_{n}^* \operatorname{Prim}_{n,k} \to 0.
\]
 Both copies of $\Ql \left(- \frac{n-2}{2} \right)$ are generated by the hyperplane class raised to the power $\frac{n-2}{2}$, so it suffices to check that this power of the hyperplane class is preserved by $i^*$. But this  follows from the preservation of the hyperplane class under pullback and the compatibility of pullback with cup product. 
On the other hand,   
  if $n$ is odd, then the exact sequence specialises to 
    \[ 
0
\to
 h_{n-1}^* \operatorname{Prim}_{n-1,k} 
\to
   R^{n-1}  pr_{\mathcal M, X !} \Ql 
\to 
h_{n}^* \operatorname{Prim}_{n,k} \oplus  \Ql \left(- \tfrac{n-1}{2} \right)
\to 
\Ql \left(- \tfrac{n-1}{2} \right).
\]
Once again, the result follows on checking the compatibility of the pullback $i^*$ with powers of the hyperplane class. \end{proof}
  
  \begin{lemma}\label{hypersurface-negation} Let $k,n\in \ZZ$ be such that  $k\geq 2$ and $ n \geq 3$. Assume that $(k,n) $ is not $(3,3)$ or $(3,4)$. Then there exists an element in the geometric monodromy group of $R^{n-1} pr_{\mathcal M,X!}  \Ql$ which acts on the stalk at the geometric generic point by some element with all eigenvalues $-1$.  \end{lemma}
  
Recall that the geometric monodromy group refers to the Zariski closure of the image of the geometric fundamental group inside of the general linear group acting on the stalk at some chosen point, and is independent, up to conjugacy, of the choice of point.

  \begin{proof}[Proof of Lemma \ref{hypersurface-negation}]
  It follows from  Lemma \ref{monodromy-of-hypersurfaces} that the lisse sheaf  $R^{n-1} pr_{\mathcal M,X!}  \Ql$ is an extension of  $h_{n}^* \operatorname{Prim}_{n,k}$ by $h_{n-1}^* \operatorname{Prim}_{n-1,k}$. It therefore  suffices to show that the geometric mondromy group of  $h_{n}^* \operatorname{Prim}_{n,k} \oplus h_{n-1}^* \operatorname{Prim}_{n-1,k}$ contains an element which acts by the scalar $-1$, since  any lift of this element to the monodromy group of the extension will act as the scalar $-1$ on both sheaves and thus act with all eigenvalues $-1$ on the extension. 
  
 Let us first check that $h_n$ and $h_{n-1}$ are smooth maps with geometrically connected generic fiber. To prove this for $h_n$, first observe that the condition that the restriction of a smooth hypersurface to the hyperplane at infinity be smooth is an open condition. The map $h_n$ is the composition of this open immersion into the moduli space of smooth hypersurfaces with the universal family whose fiber over any given hypersurface is its restriction to the divisor at infinity. Because, by assumptions, these restrictions to the divisor at infinity are smooth, $h_{n}$ is a smooth morphism. For $h_{n-1}$ this is because, given any degree $k$ leading terms, the possible ways of extending them to a polynomial of degree $k$ defining a smooth hypersurface form an open subset of an affine space. Because both $\mathcal H_{n,k}$ and $\mathcal H_{n-1,k}$ are normal, both $h_n$ and $h_{n-1}$ give surjections on the \'{e}tale fundamental group. Hence the monodromy groups on $\mathcal M_{n,k}$ of $h_{n}^* \operatorname{Prim}_{n,k}$ by $h_{n-1}^* \operatorname{Prim}_{n-1,k}$ are equal to the monodromy groups of $\operatorname{Prim}_{n,k}$ and $\operatorname{Prim}_{n-1,k}$.

 Let $n_1$ be whichever of $n$ or $n-1$ is odd, and let $n_2$ be whichever of $n$ or $n-1$ is even. Let $N_1$ be rank of $\operatorname{Prim}_{n_1,k}$ and let $N_2$ be the rank of $\operatorname{Prim}_{n_2,k}$. Then 
 it follows from \cite[Theorem 11.4.9]{KatzSarnak} that 
 the geometric monodromy of $\operatorname{Prim}_{n_1,k}$ is $O_{N_1}$ unless $(k,n_1)= (3,3)$, and the geometric monodromy of $\operatorname{Prim}_{n_2,k}$ is $Sp_{N_2}$.  (Note that the exceptional case $(k,n_1)=(3,3)$ occurs only if $(k,n)=(3,3)$ or $(3,4)$.)
The monodromy group of $h_{n}^* \operatorname{Prim}_{n,k} \oplus h_{n-1}^* \operatorname{Prim}_{n-1,k}$ is a subgroup of $O_{N_1} \times Sp_{N_2}$, whose projection to both $O_{N_1}$ and $Sp_{N_2} $ is surjective.  If $k=2$ then $N_1=1, N_2=0$, and so the subgroup must be $O_1= \pm 1$ and clearly contains an element acting by $-1$. Thus we may assume $k>2$.

By Goursat's lemma, any proper subgroup of $O_{N_1} \times Sp_{N_2} $ whose projections are both surjective arises from an isomorphism between a nontrivial quotient of $O_{N_1}$ and a nontrivial quotient of $Sp_{N_2}$. Any nontrivial quotient of $Sp_{N_2} $ is either $Sp_{N_2}$ itself or its adjoint form $Sp_{N_2} / \langle \pm 1\rangle$. For either of these to be isomorphic to a quotient of the orthogonal group, there must be an exceptional isomorphism between the symplectic and orthogonal Lie algebras, which only occurs if $(N_1,N_2) =(2,3)$ or $(4,5)$. But we have $N_1= (1-1/k ) ( (k-1)^{n_1} +1)  $. This is an increasing function of $k$ and $n_1$, and we have $k \geq 3, n_1 \geq 3$ so that $N_1 \geq (2/3) ( 2^3+1) = 6$ and thus exceptional isomorphisms cannot occur.
We have therefore shown that the monodromy group of $h_{n}^* \operatorname{Prim}_{n,k} + h_{n-1}^* \operatorname{Prim}_{n-1,k}$ is simply  $O_{N_1} \times Sp_{N_2}$. Since both $O_{N_1}$ and $Sp_{N_2}$ contain an element acting by the scalar $-1$, their product does as well. 
   \end{proof}

   We are now ready to reveal the main technical result behind the proof of Proposition~\ref{odd-vanishing}.
Recall that $pr_{\mathcal M,m}$ is defined to be the projection from $ \mathcal M_{n,k} \times  (A^{kd}_m - A^{kd}_{m-1})$ to $\mathcal M_{n,k}$.
   Viewing 
 $S_{d,f}$ as a complex of sheaves on $\mathcal M_{n,k} \times \mathbb A^{kd}$, we 
have the following result.

  \begin{lemma}\label{redo-everything}  For all $i \in \mathbb Z$ we have an isomorphism  
   \[ R^i pr_{\mathcal M,m ! } S_{d,f} \cong \left(H^{a}_c(\Pconf_m,\Ql) \otimes  (R^{n-1} pr_{\mathcal M,X !} \Ql) ^{\otimes m} \otimes \sgn^{n-1} \right)^{S_m} (-m-n(d-m)), \] 
    where $a=i+ m(n-1)-2nd$. In particular, because the right hand side of this isomorphism is lisse, the left hand  side is as well. 
 \end{lemma}
  
  \begin{proof}
  To prove this statement  we must verify that the proofs of all results from Lemma~\ref{infinityvanishing} to Corollary \ref{degreemsmooth} work in the relative setting. For most steps, the verification involves simply replacing definitions and lemmas in \'{e}tale cohomology with their relative analogues. (We have elected  to avoid 
 systematically doing this in \S~\ref{s:major},   since   the relative versions are 
 notationally more cumbersome, albeit conceptually similar.) 
Here, we shall content ourselves with  explaining all the subtleties that occur in this process.
  
 The proof of Lemma \ref{infinityvanishing} barely needs to be modified. Indeed,  by excision, it suffices to check that the stalk of $S_{d,f}$ vanishes outside $\mathcal M_{k,d} \times U_{m}^{kd}$, and this is a statement about points and thus is exactly the same statement already checked in the proof of Lemma \ref{infinityvanishing}.
 
 The relative analogues of $\rho, \overline{e} , p_{2}'$, and $\pi_2$ are given by exactly the same formulas.  Having done this, the proof of Lemma \ref{residuereduction} is essentially identical, but involves avoiding compactly supported cohomology in favour of compactly supported pushforwards.
  
  In Lemma \ref{factorisation} we use maps defined by the same formulas in the relative settings,  and we apply proper base change and the K\"{u}nneth formula in exactly the same way. 
  
  Lemma \ref{powervanishing} is a result about stalks and so does not need to be modified at all.
  
  The analogue of  $H^{*,\mathrm{red}}_c(X, \mathbb Q_\ell)$ is the mapping cocone of the trace map $$
  Rpr_{\mathcal M,X!}  \Ql \to \Ql[-2 (n-1) ] (n-1).
  $$
  By Lemma \ref{monodromy-of-hypersurfaces}, it is equal to $R^{n-1} pr_{\mathcal M,X!}  \Ql [- (n-1)]$.  Lemma \ref{redcoh} can be done with a relative excision sequence, and then for the vanishing in top degree we can use the same calculation on stalks. 
  
  In Lemma \ref{mainterm} we no longer prove that $\overline{S}_{1,f}$ is a constant complex, but rather a complex pulled back from $\mathcal M_{n,k}$. The same proof works, however.
  
  The proof of Lemma \ref{residuecalculation} (and thus Corollary \ref{degreemcalculation}) is the same. The proof now involves the compactly supported pushforward from $\mathcal M_{k,d} \times \Pconf_m$ to $\mathcal M_{k,d}$ of the pullback of the complex $\left( R^{n-1} pr_{\mathcal M,X!}  \Ql [- (n-1)]\right)^{\otimes m}$, rather than simply the cohomology of $\Pconf_m$ with coefficients in a constant complex. In order to handle it, we apply the projection formula to deduce that this is  the complex
  $\left( R^{n-1} pr_{\mathcal M,X!}  \Ql [- (n-1)]\right)^{\otimes m}$ 
  tensored with the compactly supported pushforward from $\mathcal M_{k,d} \times \Pconf_m$ to $\mathcal M_{k,d}$ 
  of   the constant sheaf, then smooth or proper base change to deduce that the compactly supported pushforward from $\mathcal M_{k,d} \times \Pconf_m$ to $\mathcal M_{k,d}$ is the constant complex $H^*_c (\Pconf_m , \Ql)$.  This is enough to deduce that 
  $R  pr_{\mathcal M,m ! } S_{d,f}$ is equal to 
 \begin{align*}
& \left( H_c^* (\Pconf_m,\Ql) \otimes  (R^{n-1} pr_{\mathcal M,X !} \Ql) ^{\otimes m} \otimes \sgn^{n-1} \right)^{S_m}\\
&\hspace{4cm}
\times  [ -2m - 2n (d-m) - m (n-1) ] (-m-n(d-m)) .\end{align*} 

Finally,
the analogue of Lemma \ref{supersmooth} follows from Lemma \ref{monodromy-of-hypersurfaces} and this immediately gives the analogue of Corollary \ref{degreemsmooth}, which is the statement of this lemma.

To understand this isomorphism, one can also take the explicit diagram Figure \ref{fig:zig} and translate each step into the relative setting. For (1), (3), (5), and (7) this involves the relative analogues of excision, the projection formula, and the trace map. Maps (2) and (6) arise from the relative versions of Lemmas \ref{residuereduction} and \ref{residuecalculation} as discussed above.  (4) still arises by definition, and (8) now follows from Lemma \ref{monodromy-of-hypersurfaces}(1).
     \end{proof}

We now have everything in place to establish Proposition 
\ref{odd-vanishing}.
Let $\overline{pr}_{\mathcal M, d}$ be the projection from $\mathcal M_{n,k} \times A^{kd}_d$ to $\mathcal M_{n,k}$. Applying excision to the filtration of $\mathcal M_{n,k} \times  A^{kd}_d$ into the descending sequence of open sets 
$\mathcal M_{n,k}  \times (A^{kd}_d - A^{kd}_{m-1})$ for $m \in \{0,\dots,d\}$,  
we obtain a filtration of $R \overline{pr}_{\mathcal M, d! } S_{d,f}$ whose associated graded objects are $R pr_{\mathcal M, m!} S_{d,f}$,  for $0\leq m\leq d$. This filtration produces a spectral sequence. 
Restricted to any point of $\mathcal M_{n,k}$, $R \overline{pr}_{\mathcal M, d! } S_{d,f}$ is $H^*_c(  A^{kd}_d , S_{d,f})$ and  $R pr_{\mathcal M, m!} S_{d,f}$ is $H^*_c( A^{kd}_m -  A^{kd}_{m-1}, S_{d,f})$. Thus, by  restricting to a point,  we obtain the spectral sequence of Proposition \ref{majormain}.

The differentials of a spectral sequence whose first page consists of lisse sheaves $\mathcal M_{n,k}$ must be morphisms of lisse sheaves. Indeed, all of the pages are 
subquotients of the first page, which consists of lisse sheaves. 
Therefore the differentials commute with the action of $\pi_1 (\mathcal M_{k,d})$, and thus with its Zariski closure, the monodromy group. By Lemma~\ref{hypersurface-negation}, there is an element of the monodromy group that acts on $R^{n-1} pr_{\mathcal M,X !} \Ql$ with all eigenvalues $-1$. It follows from Lemma \ref{redo-everything} that it 
acts on $R pr_{\mathcal M, m! } S_{d,f}$ with all eigenvalues $(-1)^m$. The differential on the $r$th page of the spectral sequence is a map from (a subquotient of) $R^i pr_{\mathcal M, m! } S_{d,f}$ to (a subquotient of)  $R^{i+1} pr_{\mathcal M, (m+r)! } S_{d,f}$.  Because the differential commutes with this element of $\pi_1$ but maps from an $(-1)^m$-eigenspace to an $(-1)^{m+r}$-eigenspace, the differential must vanish if $r$ is odd. This therefore establishes Proposition \ref{odd-vanishing}.
    
  \section{The geometric minor arcs: geometry}
\label{s:minor1}  

We now turn to the minor arcs. 
In the Hardy--Littlewood circle method, the key step is to prove a bound for the value of the exponential sum at each point inside the minor arcs. This is then used to bound the integral of this exponential sum over the union of all minor arcs. Analogously, in our setting, it will suffice to bound the highest degree in which the stalk cohomology of $S_{d,f}$ is non-vanishing at each point of $\mathbb A^{kd} - A^{kd}_d$, which will then be used to bound the highest degree in which $H^*_c( \mathbb A^{kd} - A^{kd}_d, S_{d,f})$ is non-vanishing and thus to prove Proposition \ref{minormain}. Throughout this section, the highest degree in which a cohomology group is non-vanishing will be the analogue of the size of a sum in the usual circle method.

Classically, the way that exponential sums are bounded is via the iterative method of Weyl differencing. This  reduces the problem of bounding the exponential sum of a multivariable polynomial over a compact region to bounding the number of points where a certain multilinear form associated to that polynomial takes small values. In general this multilinear form is handled via methods from the geometry of numbers (through the ``shrinking lemma''), while in the case of  diagonal polynomials  it can be handled more directly.  
The exponential sum we are applying Weyl differencing to is a degree $k$ polynomial in $n$ variables, each variable itself a polynomial in $\mathbb F_q[T]$ of degree $d$. However, the Weyl differencing argument uses only the additive structure of the variables. As an additive group, the ring of polynomials in many variables can simply be viewed as  a vector space over $\mathbb F_q$. For this reason, in this section, it will be more convenient to view our polynomial as a polynomial of degree $k$ in $dn$ variables,
each an element of $\mathbb F_q$.
(This is similar to how, in additive combinatorics, one often takes a vector space over a finite field as a model for the integers.)

Our geometric  analogue of Weyl differencing will reduce the bound for the top degree of non-vanishing stalk cohomology of $S_{d,f}$ to bounding the top degree of non-vanishing cohomology of a variety  $V(G)$ defined over $\FF_q	$, for 
a certain   polynomial $G=G_{(b_1,\dots,b_{kd})}$ defined over $\FF_q$ and indexed by 
 $(b_1,\dots,b_{kd})\in \mathbb A^{kd}$. 
The association of the variety to a given polynomial is  governed by the following definition.

\begin{definition}\label{def:VG}
Let 
$
G\in \FF_q[x_1,\dots,x_N]$ be a  polynomial of degree $k$. Then 
  $V(G)$ is defined to be the set of 
$(\y^{(1)},\dots,\y^{(k-1)})\in (\AA^N)^{k-1}$
such that 
\begin{equation}\label{eq:pear}
\sum_{\ve_1,\dots,\ve_{k-1}\in \{0,1\}} (-1)^{\ve_1+\dots+\ve_{k-1}} 
G(\x+\ve_1 \y^{(1)}+\dots+ \ve_{k-1}\y^{(k-1)})
\end{equation}
is a constant function of $\x$.
\end{definition}

In fact, the situation is very favourable in the geometric  setting, since the top degree of non-vanishing (compactly-supported) cohomology of a variety is twice the dimension of that variety. Thus  we can express our version of Weyl differencing as bounding cohomology in terms of the dimension of $V(G)$. This is advantageous: by Lang--Weil, the problem of bounding the dimension of $V(G)$ can be reduced to bounding the number of points in $V(G)$, which is exactly the counting problem solved in the usual circle method. After this reduction, we can follow the function field version of the  classical circle method directly, rather than having to develop a geometric version. We will discuss the geometric Weyl differencing step in this section and the less geometric remainder of the problem in the next section.

\medskip
  
We continue with the convention that 
$\ell$ is an arbitrary prime and that 
$\FF_q$ is a finite field such that 
$\chaar(\FF_q)=p>k$. 
The main aim of this section is to lay down the tools for bounding the ``cohomological dimension of a polynomial'', in the following sense.

\begin{definition} \label{4.1} 
Let $G$ be a polynomial in $N$ variables $x_1,\dots,x_N$ over $\mathbb F_q$. 
We denote by  $cd(G)$  the largest $i$ such that 
$H^i_c (\mathbb A^N, \mcL_\psi(G)) \neq 0$.
\end{definition}

Here, we have begun to use the alternate notation $\mcL_\psi(G) $ for $G^* \mcL_\psi$, as it is more convenient for the remaining calculations.
As noted above, we can think of this cohomological dimension as being the size of the exponential sum associated to a polynomial $G\in \FF_q[x_1,\dots,x_N]$.
In order to prove our version of Weyl's inequality, we shall need analogues of certain self-evident facts
about the size of exponential sums. Let  $\psi$ be
 a non-trivial additive character of $\mathbb F_{q}$.
The first fact is the identity \[ \left| \sum_{ \x \in \mathbb F_q^N} \psi( G(\x) )\right| = \left|\sum_{\x \in \mathbb F_q^N} \psi ( -G(\x) )\right|,\] 
which follows since  the two sums are complex conjugates.
The second fact  is the  basic
change of variables identity
\[ \left| \sum_{\x \in \mathbb F_q^N} \psi ( G(\x)) \right|  \left| \sum_{\x \in \mathbb F_q^N} \psi ( -G(\x))  \right|  = \left |   \sum_{\x, \y \in \mathbb F_q^N} \psi( G( \x) - G(\x+\y )) \right|  .\] 
We shall prove that the notion introduced in
Definition \ref{4.1} satisfies the cohomological analogues of these two properties.

\begin{lemma}\label{conjugation} Assume that $\ell$ has even order in the multiplicative group modulo $p$. Then $cd(G(x_1,\dots,x_N))=cd(-G(x_1,\dots,x_N))$. 
\end{lemma}

\begin{proof} Since  $\ell$ has even order, some power $\ell^r$ of $\ell$ has order exactly $2$ mod $p$. Hence $\ell^r \equiv -1$ modulo $p$.
The sheaf $\mcL_\psi$ can be defined over every $\ell$-adic coefficient field that includes the $p$th roots of unity. Its definition commutes with extension by scalars from one such coefficient field to another. Using the fact that  the property $H^i_c (\mathbb A^N,  \mcL_\psi(G)) \neq 0$ commutes with such an extension, we may assume the coefficient field is $\mathbb Q_\ell(\mu_p)$.

There is an automorphism $\operatorname{Frob}_\ell^r$ of $\mathbb Q_\ell(\mu_p)$, which acts as multiplying by $\ell^r$ on $\mu_p$. Thus it acts by  sending any character valued in the $p$th roots of unity to the dual character. Applying this automorphism to $\mcL_\psi(G)$, we obtain $ \mcL_{\psi^{-1}}(G) =  \mcL_\psi(-G)$.  By functoriality of cohomology in the coefficient sheaf, we obtain a $\mathbb Q_\ell(\mu_p)$-semilinear automorphism $$H^i_c (\mathbb A^N, \mcL_\psi(G))\to H^i_c (\mathbb A^N,  \mcL_\psi(-G)).
$$
Thus one of these two groups is non-vanishing if and only if the other is. \end{proof}

\begin{lemma}\label{shearing}We have the identities 
\begin{align*}cd(G(x_1,\dots,x_N)) &+ cd(-G(x_1,\dots,x_N))\\ &= cd ( G(x_1,\dots,x_N) - G(x_{N+1},\dots,x_{2N}))\\ &= cd ( G(x_1,\dots,x_N) - G(x_1+x_{N+1},\dots,x_N + x_{2N})).
\end{align*}
 \end{lemma}

\begin{proof} 
The first identity follows from the K\"{u}nneth formula
\begin{align*}
H^i_c(\mathbb A^{2N}, ~
& \mcL_\psi( G(x_1,\dots,x_N) - G(x_{N+1},\dots,x_{2N}))) \\ 
&= H^i_c(\mathbb A^N \times \mathbb A^N, \mcL_\psi(G(x_1,\dots,x_N) \otimes \mcL_\psi (-G(x_{N+1}\dots,x_{2N}))) \\
&= \bigoplus_{j+k=i} H^j_c(\mathbb A^N, \mcL_\psi(G)) \otimes H^k_c(\mathbb A^N,\mcL_\psi(-G)). 
\end{align*}
Thus it follows that the largest $i$ where $H^i_c(\mathbb A^{2N}, \mcL_\psi( G(x_1,\dots,x_N) - G(x_{N+1},\dots,x_{2N})))$ is non-zero is equal to  the largest $j$ where $H^j_c(\mathbb A^N, \mcL_\psi(G)) \neq 0$ plus the largest $k$ where $ H^k_c(\mathbb A^N,\mcL_\psi(-G))\neq 0$.

The second identity follows since  $G(x_1,\dots,x_N) - G(x_{N+1},\dots,x_{2N})$
is related to  $G(x_1,\dots,x_N) - G(x_1+x_{N+1},\dots, x_N + x_{2N})$ via an invertible change of variables. Thus the associated cohomology groups are isomorphic, and are non-vanishing in the same degrees. 
\end{proof}

We 
 now come to record a general algebraic geometry argument, which adapts the standard analytic strategy of Weyl differencing to the task of bounding the cohomological dimension of a polynomial. Recall Definition \ref{def:VG} and the particular variety $V(G)$ that is associated to any polynomial
$
G\in \FF_q[x_1,\dots,x_N]$ of degree $k$.
The following is the main result of this section.

\begin{prop}\label{weyl} 
Assume that $\ell$ has even order in the multiplicative group modulo $p$ and let $G\in \FF_q[x_1,\dots,x_N]$  have degree $\leq k$.
Then 
$$
cd(G)\leq \frac{dd(G)+N(2^{k-1}-(k-1))}{2^{k-2}},
$$
where $dd(G)=\dim V(G)$.
\end{prop}

For some intuition about this, recall that the first step in the classical Weyl differencing argument takes the form 
\begin{align*} 
\left| \sum_{\x \in \mathbb F_q^N}\psi ( G(\x) ) \right|^2 &=  
\left| \sum_{\x \in \mathbb F_q^N}\psi( G(\x) ) \right| \left| \sum_{\x \in \mathbb F_q^N}\psi ( -G(\x) ) \right|\\ 
&=  \left| \sum_{\x, \y \in \mathbb F_q^N}\psi ( G(\x) - G(\x + \y) ) \right|\\
&
\leq \sum_{ \y \in \mathbb F_q^N } \left| \sum_{\x\in \FF_q^N}\psi ( G(\x) - G(\x + \y) ) \right| .
\end{align*} 
The key point is that, for any fixed $\y$, the degree of $G( \x) - G(\x+ \y)$ 
is one less than then the degree of $G$, 
as a polynomial in $\x$. Thus we can reduce an exponential sum of degree $k$ polynomials to many exponential sums of degree $\leq k-1$ polynomials. Iterating, we may reduce all the way to exponential sums of degree $\leq 1$, which cancel if the degree is exactly $1$ but do not cancel if the degree is $0$. The 
 relevance of  $V(G)$ is that is precisely the set where we end up in a degree $0$ function after $k-1$ iterations.

In the geometric argument, most steps proceed by close analogy to their classical counterparts. One subtlety occurs in the step where we bound the absolute value of a sum by the absolute values of each of its terms. This will be replaced by an argument in which  we stratify the base of some family and then bound the cohomological dimension of a sheaf on the total space in terms of the cohomological dimension of the strata and the cohomological dimension of the restriction of that sheaf to the fibers. This stratification is needed to handle the fact that our cohomological dimension bounds for different fibers may vary. (For the same reason we will use a further stratification  at the end of this section to deduce Proposition \ref{minormain}.)

\begin{proof}[Proof of Proposition \ref{weyl}]
The proof is by induction on $k$.
Suppose first that  $k=1$.
Then we claim that  $dd(G)=0$ if $G$ is constant and $dd(G)=-\infty$ otherwise. 
When $k=1$ there are no variables 
$\y^{(1)}, \dots, \y^{(k-1)}$.
Thus the closed set under consideration is a subset of  a point, and the alternating sum 
\eqref{eq:pear}
is simply $G(x_1,\dots,x_n)$. 
The claim follows, since it is now clear that $V(G)$ is a point if  $G$ is constant and $V(G)$ is the empty set otherwise.
If $G$ is constant we must check  that
\[cd(G) \leq  \frac{ 0+N (2^{1-1} -(1-1)) }{2^{1-2}}=2N.\] 
But this follows  from the cohomological dimension of an $N$-dimensional variety being $2N$. Moreover, we have  $cd(G) \leq - \infty$ if $G$ is non-constant, since  then  the compactly supported cohomology of the associated Artin--Schreier sheaf vanishes.

Now assume that the result is already known for polynomials of degree $k-1$. Let $G$ be a polynomial of degree $k\geq 2$. Then for any $\y_0=(y_{1},\dots,y_{N})$, the difference $G(\x) - G(\x+\y_0)$ is a polynomial of degree $k-1$ in $x_1,\dots,x_{N}$. The variety $V(G)$  in the definition of $dd(G)$ admits a map $V(G)\to \mathbb A^N$ along the coordinates $\y$, whose fiber over a point $\y_0\in \AA^N$ is the variety in the definition of $dd ( G(\x) - G(\x+\y_0))$.  Choose a stratification of $\AA^N$ whose strata $W_j$ are varieties and such that the fiber dimension of this map is constant on each stratum $W_j$. (This is possible since the fiber dimension is constructible, by  \cite[Ex.~II.3.22(e)]{hart}, and any constructible function can be made constant on a stratification.) 

Viewing
$G(\x) -G(\x+\y)$ as a polynomial in $2N$ variables $\x,\y$, 
the cohomological dimension $cd ( G(\x) -G(\x+\y))$ is defined to be the maximum $i$ such that $H^i_c(\AA^n \times \AA^n, \mathcal L_\psi(G(\x) - G(\x+\y)))\neq 0$. 
By iteratively applying excision to the chosen stratification, this is at most the supremum over strata $W_j$ of the maximum $i$ such that $H^i_c( \AA^n \times W_j, \mathcal L_\psi(G(\x) - G(\x+\y)))\neq 0$. So it suffices to bound this $i$ for all possible strata $W_j$.

Let $W_j$ be a stratum of dimension $r$. Let $\pi: \AA^n \times W_j \to W_j$ be the projection map. Then by the Leray spectral sequence  with compact supports we have 
\[H^i_c( \AA^n \times W_j, \mathcal L_\psi(G(\x) - G(\x+\y))) = H^i_c(W_j, R \pi_! \mathcal L_\psi(G(\x) - G(\x+\y))).\] 
Since  $W_j$ has cohomological dimension at most $2r$, this vanishes in degrees greater than $2r+s$, where $s$ is the 
highest degree in which the cohomology of the complex $R\pi_! \mathcal L_\psi(G(\x) - G(\x+\y))$ is non-zero. By proper base change and the induction hypothesis, we have  
\begin{align*}
s&\leq \max_{\y_0\in W_j}
cd(G(\x) - G(\x+\y_0))\\
&\leq  \max_{\y_0\in W_j}
\frac{ dd(G(\x) - G(\x+\y_0))+ N (2^{k-2}- (k-2))}{2^{k-3}} .
\end{align*} 
Now  $dd(G(\x) - G(\x+\y_0))$ is the fiber dimension of $V(G)$ over $\y_0$, which is constant on $W_j$. Hence the dimension of the inverse image of $W_j$ in $V(G)$ is equal to this fiber dimension plus $r$ by \cite[Ex.~II.3.22(e)]{hart}. Because the inverse image of $W_j$ in $V(G)$ certainly has dimension at most $dd(G)$, the fiber dimension is at most $dd(G)-r$. Thus 
$H^i_c( \AA^n \times W_j, \mathcal L_\psi(G(\x) - G(\x+\y)))$
vanishes for  \[i>2r+ \frac{ dd(G)-r + N (2^{k-2}- (k-2))}{2^{k-3}} .\] 
It follows that 
\begin{align*}
cd ( G(\x) -G(\x+\y)) &\leq 
\sup_{0\leq r\leq N}   \left( 2r+ \frac{ dd(G) - r + N (2^{k-2}- (k-2))}{2^{k-3}}\right)
\\
&= \sup_{0\leq r\leq N} \frac{ dd(G) + (2^{k-2}-1) r + N (2^{k-2}- (k-2))}{2^{k-3}}\\ 
&= \frac{ dd(G) +  N (2^{k-1}- (k-1))}{2^{k-3}}. 
\end{align*}
Applying  Lemmas~\ref{conjugation} and~\ref{shearing}, we obtain
\begin{align*}
2cd(G) 
&= cd(G) + cd(-G) \\
&=  cd( G(\x) - G(\x+\y)) \\
&\leq 
\frac{ dd(G) + N (2^{k-1}- (k-1))}{2^{k-3}}.\end{align*}
The proposition follows on dividing both sides by $2$.
\end{proof}

Assume that $k>1$ and that the leading (degree $k$) terms 
of $G$ are 
$$
G_0(\x)=\sum_{i_1,\dots,i_k=1}^N a_{\mathbf{i}} x_{i_1}\dots x_{i_k},
$$
for coefficients $a_{\mathbf{i}}\in \FF_q$
which are symmetric  in the indices.
Then, 
since $\mathrm{char}(\FF_q)>k$, one confirms that 
$V(G)$ is the  set of 
$(\y^{(1)},\dots,\y^{(k-1)})\in (\AA^N)^{k-1}$ for which  
$$
\sum_{i_1,\dots,i_{k-1}=1}^N a_{i_1,\dots,i_{k-1},i} y_{i_1}^{(1)}\dots y_{i_{k-1}}^{(k-1)}=0, \quad \text{for all $1\leq i\leq N$.}
$$
In particular $V(G)=V(G_0)$ is an algebraic subvariety of $(\AA^{N})^{k-1}$.
Using Lang--Weil  we may transform the problem of bounding 
$dd(G)$ into a counting problem, as  follows.

\begin{lemma}\label{langweil} 
Let $G\in \FF_q[x_1,\dots,x_N]$ be a polynomial of degree $k$.
Suppose that there exists a constant $c_G>0$ such that 
$$
\#V(G)(\FF_{q^r}) \leq c_G q^{{rD}},
$$
for all positive integers $r$.
Then $dd(G)\leq D$.
\end{lemma}

\begin{proof} This follows by applying the Lang--Weil estimates \cite{LW}, 
since over any finite field $\FF_{q_0}$ where the variety has a geometrically irreducible component of dimension $dd(G)$, its number of points is at least $q_0^{dd(G)} (1+o(1)) $. 
If we have $q_0^{dd(G)}(1+o(1)) \leq c_G q_0^D$ for all $q_0$ then $dd(G) \leq D$. 
\end{proof}

We are going to apply these general arguments to control $\mathcal H^{i} (S_{d,f})$  (i.e.\ the $i$th cohomology sheaf of the complex $S_{d,f}$) for our polynomial $f\in \FF_q[x_1,\dots,x_n]$ of degree $k$.
For $(b_1,\dots,b_{kd})\in \mathbb A^{kd}$, we let $ G_{(b_1,\dots,b_{kd})}$ be 
the map $e$ restricted to the point $(b_1,\dots,b_{kd})$. Thus 
$ G_{(b_1,\dots,b_{kd})}$ is
the function that takes 
$$
\left(
\left(a_{0,j}, \dots,a_{d-1,j}\right)_{1\leq j\leq n}
,(b_1,\dots,b_{kd})\right)
$$
to the coefficient of $T^{-1}$ in \[ \left( \sum_{r=1}^{kd} b_r T^{-r} \right) f\left( x_1 T^d +   \sum_{i=0}^{d-1} a_{i,1} T^i, \dots, x_n T^d  +   \sum_{i=0}^{d-1} a_{i,n} T^i\right).\]
Our geometric minor arc bound amounts to   bounding 
$dd(G_{(b_1,\dots,b_{kd})})$ when the point
$(b_1,\dots,b_{kd}) $ belongs to $(\mathbb A^{kd}-A_m^{kd})(\mathbb F_{q^r})$ for given $m$. 
This is summarised as follows.

\begin{prop}
\label{arithmeticpart} Assume that the leading terms of $f$ define a smooth hypersurface in $\mathbb P^{n-1}$ and that $ \chaar(\mathbb F_{q})>k$. 
Assume that $d\geq k-1$ and let $m$ be an integer in the range 
\begin{equation}\label{eq:raison} 
d\leq m\leq 
\frac{kd}{2}. 
\end{equation}
 Then for all $r\geq 1$ and all  $(b_1,\dots,b_{kd}) 
\in (\mathbb A^{kd}-A_m^{kd})(\mathbb F_{q^r})$, we have 
\[dd(G_{(b_1,\dots,b_{kd})})  \leq dn(k-1)  -
n\left\lfloor \frac{m}{k-1}\right\rfloor.\]
\end{prop}

Using Lemma~\ref{langweil} as a base, we will tackle the proof of 
Proposition~\ref{arithmeticpart} through the sort of  counting arguments that feature in the usual circle method over  $\FF_q(T)$.
We defer the proof  to \S~\ref{s:minor-2} and instead proceed to show how it yields Proposition~\ref{minormain}.

  \begin{prop} 
 \label{prop:banana} 
 Assume that the leading terms of $f$ define a smooth hypersurface in $\mathbb P^{n-1}$, that $ \chaar(\mathbb F_{q})=p>k$, and that $\ell$ has even order mod $p$. Let $m\geq d\geq k-1$.
 Then $\mathcal H^{i} (S_{d,f})$ vanishes outside $A_m^{kd}$ provided that  
 \[
 i> 
 n  \left(2d -
\left\lfloor \frac{m}{k-1}\right\rfloor \frac{1}{2^{k-2}}
\right)
.\] 
 \end{prop}
 
 \begin{proof} 
 Let $i\in \NN$ be in the range recorded in the statement.
We first prove that the stalk of $\mathcal H^{i} (S_{d,f})$ vanishes at  $(b_1,\dots,b_{kd}) \in (\mathbb A^{kd}-A_m^{kd})(\mathbb F_{q^r})$, for every $r\geq 1$. Now it follows  
 immediately from the definition of $S_{d,f}$, by applying the proper base change theorem, that the stalk of $\mathcal H^i(S_{d,f})$ at a point $(b_1,\dots,b_{kd}) \in \mathbb A^{kd}(\FF_{q^r})$  is equal to 
 $
 H^i_c\left( (\mathbb A^d)^n_{\overline{\mathbb F}_{q^r}},  \mcL_\psi(G_{(b_1,\dots,b_{kd})}) \right).
 $ 
But, by the definition of  
$cd(G_{(b_1,\dots,b_{kd})})$, this vanishes for $i> cd(G_{(b_1,\dots,b_{kd})})$. 
Combining  Propositions~\ref{weyl} and ~\ref{arithmeticpart}, and taking $N=dn$,  we therefore deduce that 
\begin{align*}
cd(G_{(b_1,\dots,b_{kd})}) 
&\leq \frac{dn(k-1) -n\lfloor \frac{m}{k-1}\rfloor
 + dn (2^{k-1} - (k-1)) }{2^{k-2}} \\
&=n  \left(2d -
\left\lfloor \frac{m}{k-1}\right\rfloor \frac{1}{2^{k-2}}
\right)\\
&<i.
\end{align*}
Hence the stalk of $\mathcal H^{i} (S_{d,f})$ does indeed vanish.

The sheaf $\mathcal H^i(S_{d,f})$ is a constructible sheaf defined over $\mathbb F_{q}$. Hence the set where its stalks are non-vanishing is a constructible set defined over $\mathbb F_{q}$. We have shown that, for all finite fields $\mathbb F_{q^r}$, the set has no $\mathbb F_{q^r}$-points outside $A_m^{kd}$. It follows that the set has no points outside $A_m^{kd}$ at all, and that $\mathcal{H}^{i} (S_{d,f})$ vanishes outside $A_m^{kd}$, as desired. 
\end{proof}

We are now ready to deduce the  vanishing result for 
$H^i_c( \mathbb A^{kd} - A_{d}^{kd}, S_{d,f})$ that is recorded in 
Proposition~\ref{minormain}.
 Assume that the leading terms of $f$ define a smooth hypersurface in $\mathbb P^{n-1}$, that $ \chaar(\mathbb F_{q})=p>k$,  that $\ell$ has even order mod $p$ and that 
 $n>  2^k(k-1).$  
Appealing to Lemma~\ref{majorarcproperties}, we clearly have 
$$
\AA^{kd}-A_d^{kd}=\bigsqcup_{m= d}^{m_0} \left(A_{m+1}^{kd}-A_{m}^{kd}\right),
$$
for $m_0=\lfloor \frac{kd}{2}\rfloor-1.$
Applying excision to the increasing chain of closed subsets $A_m^{kd}$, we see that   $H^i_c( \mathbb A^{kd} - A_{d}^{kd}, S_{d,f})=0$ provided that 
for each  integer $m\in [d,m_0]$ 
we are able to show that  $H^i_c( A_{m+1}^{kd} - A_{m}^{kd}, S_{d,f})=0$.
 
On $A_{m+1}^{kd} - A_{m}^{kd}$, the 
cohomology sheaf of  $S_{d,f}$ vanishes in degrees greater than 
$$ 
n  \left(2d -
\left\lfloor \frac{m}{k-1}\right\rfloor \frac{1}{2^{k-2}}
\right), 
$$
by Proposition~\ref{prop:banana}. By the spectral sequence for the cohomology of a complex, together with the fact that the cohomology of a variety of dimension $2(m+1)$ with coefficients in any sheaf vanishes in degrees $>4(m+1)$, it follows that 
$H^i_c( A_{m+1}^{kd} - A_{m}^{kd}, S_{d,f})=0$
provided that 
\begin{align*}
i
&>
n  \left(2d -
\left\lfloor \frac{m}{k-1}\right\rfloor \frac{1}{2^{k-2}}
\right)
+ 4(m+1)\\
&\geq 
n  \left(2d -
\left\lfloor \frac{m}{k-1}\right\rfloor \frac{1}{2^{k-2}}
\right)
+ 4(k-1)\left\lfloor \frac{m}{k-1}\right\rfloor+4.
\end{align*}
Since $n>2^k(k-1)$, 
the right hand side is a non-increasing  function of $m$. Thus it suffices to check it for $m=d$, 
which thereby completes the proof of 
Proposition~\ref{minormain}.

\section{The geometric minor arcs: arithmetic}\label{s:minor-2}

This section is devoted to the remaining task of proving  Proposition~\ref{arithmeticpart}.
Assume that  the leading terms $f_0$ of $f$
define a smooth projective hypersurface $Z\subset \mathbb P^{n-1}$. Let 
  $P=(x_1: \ldots :x_n)\in Z(\FF_q)$ and 
suppose that 
$$
f_0(x_1,\dots,x_n)=\sum_{j_1,\dots,j_k=1}^n c_{\mathbf{j}} x_{j_1}\dots x_{j_k},
$$
for symmetric coefficients $c_{\mathbf{j}}\in \FF_q$ (i.e.\  $c_{\mathbf{j}}=c_{\sigma(\mathbf{j})}$ for any  $\sigma\in S_k$).

We need to investigate 
$G_{(b_1,\dots,b_{kd})}$ for 
$(b_1,\dots,b_{kd}) 
\in (\mathbb A^{kd}-A_m^{kd})(\mathbb F_{q^r})$ for any $r\in \NN$ and any $m$ in the range \eqref{eq:raison}.
It will be convenient to redefine $q^r$ to be $q$.
Writing 
$
\alpha=\sum_{r=1}^{kd}b_r T^{-r},
$
the function $G_{(b_1,\dots,b_{kd})}$ is equal to the coefficient of 
$T^{-1}$  in  $\alpha f(g_1,\dots,g_n)$, where
$
g_j(T)=
x_jT^d+
\sum_{i=0}^{d-1} a_{i,j}T^i.
$
Let us set 
$
\a=\left(a_{0,j}, \dots,a_{d-1,j}\right)_{1\leq j\leq n},
$
a vector that has 
$N=dn$ components. 
It is now clear that 
$$
G_{(b_1,\dots,b_{kd})}=\sum_{j_1,\dots,j_k=1}^n \sum_{i_1,\dots,i_k=0}^d
d_{\mathbf{j},\mathbf{i}} a_{i_1,j_1}\dots a_{i_k,j_k}=F(\a)
$$
say, where $d_{\mathbf{j},\mathbf{i}}=c_{\mathbf{j}} b_{i_1+\dots + i_k+1}$ has symmetric indices and we follow the convention that  $a_{d,j}=x_j$ for 
$1\leq j\leq n$. In particular,  $F(\a)$ is a degree $k$ polynomial in $\a$ with leading terms
$$
F_0(\a)=\sum_{j_1,\dots,j_k=1}^n \sum_{i_1,\dots,i_k=0}^{d-1}
d_{\mathbf{j},\mathbf{i}} 
a_{i_1,j_1}\dots a_{i_k,j_k}.
$$
Hence $dd(G_{(b_1,\dots,b_{kd})})=\dim V(F_0)$.

Writing  $N=dn$, we see that $V(F_0)$ is the set of 
$(\a^{(1)},\dots,\a^{(k-1)})\in (\AA^N)^{k-1}$ for which 
$$
\sum_{j_1,\dots,j_{k-1}=1}^n \sum_{i_1,\dots,i_{k-1}=0}^{d-1}
d_{(j_1,\dots,j_{k-1},j),(i_1,\dots,i_{k-1},i)}  
a_{i_1,j_1}^{(1)}\dots a_{i_{k-1},j_{k-1}}^{(k-1)}=0,
$$
for all $0\leq i \leq d-1$ and all $1\leq j\leq n$.
Let  $\mathcal{N}$ be the number of 
$(\a^{(1)},\dots,\a^{(k-1)})\in (\FF_q^N)^{k-1}$ for which this system of equations holds.  
According to Lemma~\ref{langweil},
in order to prove Proposition~\ref{arithmeticpart} 
it will suffice to show that there is a constant $c=c(d,k,n)$ such that 
\begin{equation}\label{eq:LW}
\limsup_{q\to \infty} q^{-D}\mathcal{N}\leq c,
\end{equation}
with $D=dn(k-1)-n\lfloor \frac{m}{k-1}\rfloor.$

We shall estimate $\mathcal{N}$ by reinterpreting it as a problem about 
counting $\FF_q[T]$-points on an appropriate variety. On 
$\FF_q(T)$ we have a non-archimedean absolute value $|\cdot|$, which is extended to $\FF_q((1/T))$ 
and vectors 
in the obvious way.
For any $\beta=\sum_{i\leq M}b_i T^i \in \FF_q((1/T))$
 we   put $\|\beta\|=|\sum_{i\leq -1}b_it^i|$.

Associated to $f_0$ are the multilinear forms
$$
\Psi_j(\h^{(1)},\dots,\h^{({k-1})})
=k!\sum_{j_1,\dots,j_{k-1}=1}^n c_{j_1,\dots,j_{k-1},j} 
h_{j_1}^{(1)}\dots h_{j_{k-1}}^{(k-1)},
$$
for $1\leq j\leq n$.
As above  we write 
$\alpha=\sum_{r=1}^{kd}b_rT^{-r}\in \FF_q((1/T))$ for the point associated to the vector $(b_1,\dots,b_{kd}).$
The following result underpins  our investigation of $\mathcal{N}$.

\begin{lemma}\label{lem:N=N}
Assume that 
$\mathrm{char}(\FF_q)>k$. Then  $\mathcal{N}=N(\alpha)$, where
$$
N(\alpha)=
\#\left\{ \underline{\u}\in \FF_q[T]^{(k-1)n}: 
\begin{array}{l}
|\u^{(1)}|,\dots,|\u^{(k-1)}|<q^d\\
\|\alpha \Psi_j(\underline{\u}) \|<q^{-d} ~\text{ for $1\leq j\leq n$}
\end{array}{}
\right\}
$$
and $\underline{\u}=(\u^{(1)},\dots,\u^{(k-1)})$.
\end{lemma}

\begin{proof}
We write 
$u_{j}^{(i)}=\sum_{i=0}^{d-1} z_{i,j}^{(i)}T^i$ for $1\leq j\leq n$ and 
$1\leq i\leq k-1$.
Then 
$\alpha\Psi_j(\underline{\u})$ is equal to
$$
k!
\sum_{r=1}^{kd}b_r \sum_{j_1,\dots,j_{k-1}=1}^n c_{j_1,\dots,j_{k-1},j} 
 \sum_{i_1,\dots,i_{k-1}=0}^{d-1}
z_{i_1,j_1}^{(1)}\dots
z_{i_{k-1},j_{k-1}}^{(k-1)}T^{i_1+\dots+i_{k-1}-r},
$$
for $1\leq j\leq n$. 
The condition $\|\alpha \Psi_j(\underline{\u}) \|<q^{-d}$ is equivalent to demanding that the coefficient of $T^{-i-1}$ vanishes for 
$0\leq i \leq d-1$. 
Let $N=dn$.  Since $\mathrm{char}(\FF_q)>k$, we therefore see that 
$N(\alpha)$ is equal to the number of 
$(\z^{(1)},\dots,\z^{(k-1)})\in (\FF_q^N)^{k-1}$ for which 
$$
\sum_{j_1,\dots,j_{k-1}=1}^n 
 \sum_{i_1,\dots,i_{k-1}=0}^{d-1}
 c_{j_1,\dots,j_{k-1},j}  b_{i_1+\dots+i_{k-1}+i+1}
z_{i_1, j_{1}}^{(1)}\dots
z_{i_{k-1},j_{k-1}}^{(k-1)}=0,
$$
for $1\leq j\leq n$ and $0\leq i \leq d-1$.
The lemma follows on recalling that 
$d_{\mathbf{j},\bell}=c_{\mathbf{j}} b_{i_1+\dots + i_k+1}$.
\end{proof}

The quantity $N(\alpha)$ should be familiar to experienced practitioners of the circle method and we shall  adapt arguments found in   \cite{BV'} to estimate it. 
Our goal is to establish the following result.

\begin{proposition}\label{t:upper}
Assume that $d\geq k-1$ and 
$\mathrm{char}(\FF_q)>k\geq 3$. 
Let $m$ be an integer in the range \eqref{eq:raison}.
Let 
$(b_1,\dots,b_{kd})\in \AA^N-A_m^{kd}$
and 
put 
 $\alpha=\sum_{r=1}^{kd}b_r T^{-r}$.
Then 
there exists a constant $c_{d,k,n}>0$, independent of $q$, such that 
$
N(\alpha)\leq  c_{d,k,n} q^{dn(k-1)-n\lfloor \frac{m}{k-1}\rfloor}.
$
\end{proposition}

Applying Lemma~\ref{lem:N=N}, we deduce from Proposition \ref{t:upper} 
 that we may take $$
 D\leq dn(k-1)-n \left\lfloor \frac{m}{k-1}\right\rfloor
 $$ 
 in \eqref{eq:LW}.
The statement of Proposition~\ref{arithmeticpart} follows.

\begin{remark} This is the first place where our argument would simplify dramatically if we restricted to the case of a diagonal form like $f(\x) = \sum_{i=1}^n x_i^k$. Indeed, we could then perform Weyl differencing in each variable $x_i$ separately, and so reduce to bounding 
\[ \#\left\{ (u^{(1)},\dots, u^{(k-1)})\in \FF_q[T]^{(k-1)}: 
\begin{array}{l}
|u^{(1)}|,\dots,|u^{(k-1)}|<q^d\\
\|k!\alpha u^{(1)} \dots u^{(k-1)}  \|<q^{-d} 
\end{array}{}
\right\}, \]
which is the analogue of $N(\alpha)$ in the case $n=1$.
To handle the contribution from non-zero $u^{(i)}$
one proceeds by  collecting together the 
terms in which $k! u^{(1)} \dots u^{(k-1)} $ has a particular value $u\in \FF_q[T]$ say, exploiting the fact  that the number of ways in which this can be done is bounded efficiently in terms of the (function field analogue of the) divisor function evaluated at $u$.  Finally, 
the  Diophantine approximation properties of $\alpha$ easily lead to a bound for the number of $u\in \FF_q[T]$ such that  $|u|\leq q^{(d-1)(k-1)}$ and  $\|\alpha u\|<q^{-d}$.
\end{remark}

We now turn to the proof of Proposition~\ref{t:upper}, for which we shall require a technical 
refinement of Davenport's
``shrinking lemma'' in the function field setting. 
A
 lattice in $\FF_q((T^{-1}))^N$ is a set of points of the form ${\bf x} = \Lambda \bf u$ where $\Lambda$ is an $N \times N$ invertible matrix over $\FF_q((T^{-1}))$ and $ {\bf u}$ runs over elements of $\FF_q[T]^n$. Given a lattice $\Lambda$, the adjoint lattice is defined as the lattice associated to the inverse transpose matrix $\Lambda^{-T}$.
The following result is a 
refinement of \cite[Lemma 4.2]{BV'}.

\begin{lemma}\label{new-geometry} 
Let $\gamma$ be a symmetric $n \times n$ matrix with entries in $\FF_q((T^{-1}))$. 
Let $a, c,s\in \mathbb{Z}$ such that $c >0$ and $s \geq 0$.
Let $N_{\gamma,a,c}$ be the number of ${\bf x} \in \FF_q[T]^n$ such that $| {\bf x} | < q^a$ and $\| \gamma {\bf x} \| < q^{-c}$.  Then
\[ 
\frac{N_{\gamma,a,c}}{N_{\gamma,a-s,c+s}} \leq q^{ns + n  \max\{ \lfloor \frac{a-c}{2} \rfloor, 0 \}}. \] \end{lemma}

\begin{proof}
The bound is trivial when $a\leq 0$ since then the left hand side is $1$. Hence we may assume that $a>0$ in what follows.  It will be convenient to adopt the notation $\hat R=q^R$ for any $R\in \RR$.
Let  
\[ \Lambda_{a,c} = \begin{pmatrix} t^{-a} I_n & 0 \\ t^c \gamma & t^c I_n \end{pmatrix} ,
\] 
so that 
\[ 
\Lambda_{a,c}^{-T} = \begin{pmatrix} t^{a} I_n & 
-t^{a} 
\gamma \\ 0 & t^{-c} I_n \end{pmatrix}. \] 
We note that  
\[ t^{c-a} 
\Lambda_{a,c}^{-T} 
=  \begin{pmatrix} 
t^{c} I_n 
 & 
 - t^{c} \gamma 
  \\ 0 & t^{-a} I_n \end{pmatrix} = 
\begin{pmatrix} 0 & I_n \\ - I_n & 0 \end{pmatrix} \Lambda_{a,c}  \begin{pmatrix} 0 & I_n \\ - I_n & 0 \end{pmatrix} ^{-1} .\] 
Let $\hat{R}_1\leq  \dots \leq \hat{R}_{2n}$ denote the successive minima of the lattice corresponding to $\Lambda_{a,c}$.  Then \[ 
q^{c-a} / \hat{R}_{2n}\leq  \dots\leq  q^{c-a} / \hat{R}_1\] are the successive minima of the lattice corresponding to $t^{c-a} 
\Lambda_{a,c}^{-T} 
 $. 
Since the lattices are equal up to left and right multiplication by a matrix in $\mathrm{GL}_{2n}(\FF_q)$, 
we must have  
$\hat R_i=q^{c-a}/ \hat{R}_{2n+1-i}$  
for all $1\leq i\leq 2n$. Taking $i=n+1$ we deduce that 
$$
q^{ \lceil \frac{c-a}{2} \rceil} \leq \hat{R}_{n+1}.
$$ 

Now $N_{\gamma,a,c}$ is simply the number of vectors in the lattice 
$\Lambda_{a,c}$
of norm $<1$, while $N_{\gamma,a-s,c-s}$ is the number of norm $<q^{-s}$.  Hence, 
as established in Lee \cite[Lemma 3.3.5]{lee}, we have 
$$
N_{\gamma,a,c} = \prod_{i=1}^{2n} \max\{1, \hat{R}_{i}^{-1}\}\quad 
\text{ and 
}\quad N_{\gamma,a-s,c+s}= \prod_{i=1}^{2n}  \max\{1, q^{-s} \hat{R}_i^{-1}\}.
$$ 
Dividing term by term, we see that each $i$ contributes at most $q^{s}$ and each $i\geq n+1$ contributes at most $q^{ \max\{ \lfloor \frac{a-c}{2} \rfloor, 0\}}$. Thus the total contribution is at most $q^{ n    s + n  \max\{ \lfloor \frac{a-c}{2} \rfloor, 0\} }$, as desired. 
\end{proof}

Recalling the definition of $N(\alpha)$ from Lemma \ref{lem:N=N} it follows from 
Lemma \ref{new-geometry} that 
$$
N(\alpha)\leq q^{(k-1)ns} N_s(\alpha)
$$
for any integer $s\geq 0$, where
$$
N_s(\alpha)=
\#\left\{ \underline{\u}\in \FF_q[T]^{(k-1)n}: 
\begin{array}{l}
|\u^{(1)}|,\dots,|\u^{(k-1)}|<q^{d-s}\\
\|\alpha \Psi_j(\underline{\u}) \|<q^{-d-(k-1)s} ~\text{ for $1\leq j\leq n$}
\end{array}{}
\right\}.
$$
Suppose that we are given a vector
$(b_1,\dots,b_{kd}) 
\in (\mathbb A^{kd}-A_m^{kd})(\mathbb F_{q})$, for an integer $m$ in the range 
\eqref{eq:raison}.
Let $\alpha=\sum_{r=1}^{kd}b_r T^{-r}$ be the corresponding point in $\TT$.
Suppose that $\underline\u$ is counted by $N_s(\alpha)$, but is such that
$\Psi_j(\underline{\u})\neq 0$ for some $1\leq j\leq n$.
Putting $r=\Psi_j(\underline{\u})$, it follows that $|r|\leq q^{(d-s-1)(k-1)}$. We can ensure that 
$|r|\leq q^m$ by demanding that 
$$
s(k-1)\geq (d-1)(k-1)-m
$$
Next, let $a$ be the integer part of $\alpha \Psi_j(\underline{\u}) $. Then we have 
$$
|r\alpha-a|=\|\alpha \Psi_j(\underline{\u})\|<q^{-d-(k-1)s}.
$$
On the other hand, 
since $(b_1,\dots,b_{kd}) \not\in A_m^{kd}$,  Lemma~\ref{majorarcproperties}(4) implies that 
$
|r\alpha-a|\geq q^{-dk+m},
$
since $|r|\leq q^m$.
Thus we arrive at a contradiction if 
$$
s(k-1)\geq d(k-1)-m.
$$
Recalling that $s$ is also required to be a non-negative integer,  we are clearly led to make the choice 
$$
s=\max\left\{0, d-\left\lfloor \frac{m}{k-1}\right\rfloor \right\}=
d-\left\lfloor \frac{m}{k-1}\right\rfloor, 
$$
since \eqref{eq:raison} implies that 
$$
\left\lfloor \frac{m}{k-1}\right\rfloor\leq  \frac{m}{k-1} \leq \frac{kd}{2(k-1)}\leq d.
$$

With our choice of $s$ it now follows that 
$$
N_s(\alpha)=
\#\left\{ \underline{\u}\in \FF_q[T]^{(k-1)n}: 
\begin{array}{l}
|\u^{(1)}|,\dots,|\u^{(k-1)}|<q^{d-s}\\
 \Psi_j(\underline{\u})=0 ~\text{ for $1\leq j\leq n$}
\end{array}{}
\right\}.
$$
Since  $Z$ is smooth,
the 
system of equations $\Psi_j=0$ defines an affine variety $V$ of dimension 
at most $(k-2)n$. 
To see this, we note that the intersection of $V$ with the diagonal $\Delta=\{\underline\u\in \AA^{(k-1)n}: 
\u^{(1)}=\dots=\u^{(k-1)}\}$ is contained in the singular locus of $f_0=0$ and so has affine dimension $0$. The claim follows on noting that 
$$0=\dim(V\cap \Delta)\geq \dim V+\dim \Delta -(k-1)n=\dim V-(k-2)n.
$$
Thus  \cite[Lemma~2.8]{BV} implies that the 
there are 
$\ll q^{ (k-2)n(d-s)}$ choices for $\underline\u$, with an implied constant that depends only on $k$ and $n$. We have therefore shown that 
$$
N(\alpha)\ll 
q^{(k-1)ns} \cdot  q^{ (k-2)n(d-s)}
= q^{dn(k-1)-n(d-s)}
= q^{dn(k-1)-n\lfloor \frac{m}{k-1}\rfloor}.
$$
The statement of  Proposition~\ref{t:upper} is now clear.

\section{Topological interpretation}

As described in \S~\ref{s:intro}, 
Theorem~\ref{main2} describes a kind of homological stabilisation 
phenomenon.  In this section we draw comparisons 
with  work of Segal \cite{segal} on the moduli space of degree $d$ maps $\mathbb P^1 \to \mathbb P^n$ over $\mathbb C$ that send the point $\infty$ of $\mathbb P^1$ to a fixed point of $\mathbb P^n$, where as in our setting
 there is no natural morphism from the space of degree $d$ maps to the space of degree $d+1$ maps. Because $\mathbb P^1(\mathbb C)$ is simply the sphere $S^2$, the space of degree $d$ maps naturally embeds into the space of based continuous maps $S^2 \to \mathbb P^n(\mathbb C)$, which is the based double loop space $\Omega^2 \mathbb P^n(\mathbb C)$. Segal showed that this embedding is a homotopy equivalence up to dimension $d(2n-1)$, and in particular is an isomorphism on the first $d(2n-1)$ homology groups \cite[Prop.~1.2]{segal}.  A lot of subsequent work has been directed at proving similar results for spaces of maps from $\mathbb P^1$ to other algebraic varieties. 

The situation in our case is somewhat different, because we are looking at maps between non-compact varieties and our base points do not lie in the varieties but rather on the boundary. However, we still obtain a natural map to a double loop space, and we conjecture that a similar stabilisation result holds.

Let  $X$ be a smooth affine hypersurface over $\mathbb C$ with smooth projective closure $\overline{X}$. Let $d\in \NN$ be a natural number and let $P=(x_1: \ldots : x_n : 0)\in \overline{X}-X$.  We let $Hom_{d,P}(\mathbb C,X)$ be the space parameterizing continuous (but not necessarily holomorphic) maps $\mathbb P^1(\mathbb C) \to \overline{X}$ such that the point $\infty$ is sent to $P$, with every other point sent to $X$, and such that the map, expressed in terms of a local coordinate $ z$ near $\infty$, has the form 
\begin{equation}\label{eq:local}
(x_1 + O(|z|): \ldots : x_n + O(|z|) : z ^d + O( |z|^{d+1})).
\end{equation}
Then there is a map $\Mor_{d,P}(\mathbb A^1,X) \to Hom_{d,P} (\mathbb C, X)$, where we can check the condition on the local coordinate at $\infty$ by using $z =T^{-1}$.  We have been led to formulate the following conjecture.

\begin{conjecture}\label{stabilisation-conjecture} 
Assume that $d \geq k-1 \geq 2$.  If \[ 0 \leq j <  4 \left( \left \lfloor \frac{d}{k-1} \right \rfloor \left( \frac{n}{2^k} - k + 1\right) - 1\right) \] then the pairing  \[H^{2d(n-k)- j}_c ( \Mor_{d,P}(\mathbb A^1,X) , \mathbb Q)   \otimes H^{j} ( Hom_{d,P} (\mathbb C, X), \mathbb Q) \to \mathbb Q,\] induced by functoriality along $\Mor_{d,P}(\mathbb A^1,X) \to Hom_{d,P} (\mathbb C, X)$ and the trace map, is a perfect pairing. \end{conjecture} 

We proceed  by making the following observation. 

\begin{lemma} \label{lem:homotopic}
As long as it is nonempty, $Hom_{d,P}(\mathbb C,X)$ is 
homotopic to the based double loop space $\Omega^2 X$ of $X$. 
\end{lemma}

\begin{proof} Fix a point of $Hom_{d,P}(\mathbb C,X)$. We may as well choose the base point to lie in the image of this map. Having done this, we can define a map $\Omega^2 X \to Hom_{d,P}(\mathbb C,X)$ by gluing the fixed map $\mathbb P^1(\mathbb C) \to \overline{X}$ to an arbitrary map $S^2 \to X$ at that based point, by  using the fact that $\mathbb P^1(\mathbb C)$ is a 2-sphere and fixing a suitable map from a 2-sphere to the wedge sum of two 2-spheres. 

To obtain a homotopy inverse, we check that we can canonically deform any map $\mathbb CP^1 \to X$ which has the form 
\eqref{eq:local} near $\infty$ to our 
fixed map in a neighbourhood of $\infty$.  The  fact that we have fixed the leading coefficients makes this possible. Near this point, one of the coordinates is locally a unit, and we can divide all the coordinates by it. Because the intersection of $\overline{X}$ with $\infty$ is smooth, one of the coordinates can be written as a holomorphic function of the other coordinates, and we can drop it. Having done this, we can use the convex combinations to canonically deform any map to our fixed map.  Because the leading coefficient of the last coordinate is fixed, this convex combination will not introduce any new zeroes in a neighbourhood. 
We can then deform the map to agree with our fixed map in larger neighbourhoods of $\infty$ until it agrees on a whole half-sphere and hence can be expressed as a gluing. 
\end{proof}

We shall show that 
for $K=\CC$, Conjecture~\ref{stabilisation-conjecture} implies our earlier Conjecture~\ref{con:degenerate} on the degeneration of the spectral sequence in Theorem~\ref{main2} on the first page.
Our plan for doing this is to  calculate the dimensions of the rational cohomology groups of $Hom_{d,P} (\mathbb C, X)$, which by Conjecture~\ref{stabilisation-conjecture}, allows  us to calculate the dimensions of the rational cohomology groups of $\Mor_{d,P}(\mathbb A^1,X)$. Next, we calculate the dimensions of the cohomology groups on the first page of our spectral sequence and compare them. We show that, if any non-zero differentials existed, the dimension of  $H^{2d(n-k)- j}_c ( \Mor_{d,P}(\mathbb A^1,X) , \mathbb Q)$ would be less than its predicted value under Conjecture~\ref{stabilisation-conjecture}. Thus the 
conjecture implies that  the differentials vanish and the sequence degenerates.

This builds on (unpublished)
work of Ellenberg and Venkatesh, who used a loop space model to predict the supertrace of Frobenius on the cohomology group of a similar 
mapping space, and saw that it agreed with the main term from the circle method. Our situation differs in that we do not consider the Frobenius action on the cohomology groups but do need to understand the dimension of individual cohomology groups and not just the Euler characteristic.

\begin{definition}\label{def:e}
Given $N\in \NN$, let 
$e_k(N)$  be the unique sequence of integers such that 
$$\prod_{k=1}^\infty (1 -T^k) ^{-e_k(N)} = 1 +  (-1)^{n-1} N  T, $$ 
for a formal variable $T$.
\end{definition}

We claim that 
\begin{equation}\label{eq:inversion-formula} 
e_k(N) = - \frac{1}{k} \sum_{d \mid k } \mu\left(d\right) ((-1)^n N)^{k/d}.
\end{equation}
To check this we take  logarithms of both sides of the identity 
in Definition~\ref{def:e}. This yields
\[ -\sum_{k=1}^\infty e_k(N) \sum_{d=1}^\infty \frac{ T^{kd}}{d}   = \sum_{m=1}^\infty  ( (-1)^{n} N)^m \frac{T^m}{m}.  \]
On  extracting the coefficient of $T^m$, we obtain
\[ \frac{((-1)^{n} N)^m }{m} = -\sum_{d \mid m}  \frac{  e_{m/d}(N)}{d} 
=-\frac{1}{m} \sum_{d \mid m} (m/d) e_{m/d}(N).
\]
The M\"{o}bius inversion formula now yields
\[ - k e_k(N) =  \sum_{d \mid k } \mu\left(k/d\right) ((-1)^n N)^d
,\]
from which the claimed equality \eqref{eq:inversion-formula} follows.

The numbers $e_k(N)$ will feature prominently in our calculations of various dimensions.  We begin with the following result. 

\begin{lemma}\label{homotopy-pconf} Let $m,N\in \NN$, let $i\in \ZZ$  and let $V$ be a vector space of dimension $N$. 
Then $\dim (H^{m+i}_c( \Pconf_{m},\Ql) \otimes  V^{\otimes m} \otimes \sgn^{n-1} )^{S_m}$ is $(-1)^{mn+i}$ times the coefficient of $q^i U^m$ in $\prod_{k=1}^\infty (1 - q U^k )^{-e_k(N)}$.  \end{lemma}

\begin{proof} 
By orthogonality of characters, the dimension of the $S_m$-invariants of  $$H^{m+i}_c( \Pconf_{m},\Ql) \otimes  V^{\otimes m} \otimes \sgn^{n-1} $$ is the inner product of the characters of $S_m$ corresponding to $H^{m+i}_c( \Pconf_{m},\Ql)$ and $V^{\otimes m} \otimes \sgn^{n-1} $. Let $\chi$ be the character of $S_m$ associated to $V^{\otimes m} \otimes \sgn^{n-1} $. For a finite field $\mathbb F_q$, we may view $\chi$ as a function on squarefree polynomials of degree $m$ over $\mathbb F_q$ by evaluating it on the conjugacy class of Frobenius. This is a conjugacy class in $S_m$ with one cycle for each irreducible factor of the polynomial, of length equal to the degree of the irreducible factor. It is a special case of \cite[Theorem 3.7]{CEF} that the sum of $\chi(f)$ over all monic squarefree polynomials $f$ of degree $m$ over $\mathbb F_q$ is equal to 
\begin{align*}
\sum_i (-1)^i q^{m-i}  \langle \chi, H^{i}_c( \Pconf_{m},\Ql) \rangle 
&=\sum_i (-1)^{m-i} q^{i}  \langle \chi, H^{m-i}_c( \Pconf_{m},\Ql) \rangle  
\\
&=  \sum_i (-1)^{m-i} q^{i}  \langle \chi, H^{m+i}_c( \Pconf_{m},\Ql) \rangle,
\end{align*}
 by Poincar\'{e} duality.

Next we will compute the sum of this character $\chi$ over squarefree monic polynomials, showing it is equal as a polynomial in $q$ to $(-1)^{mn-m}$ times the coefficient of $U^m$ in $\prod_{k=1}^\infty (1 - q U^k )^{-e_k(N)}$. Because the coefficients of this polynomial are uniquely determined by its values, we will conclude that the dimensions are as stated.

First we calculate the character $\chi$ of $V^{\otimes m} \otimes \sgn^{n-1}$. We can think of $V$ as admitting a basis $v_1,\dots, v_N$, which induces a basis on $V^{\otimes m}$, on which the conjugacy class $\sigma$ will act by permutations. The trace is the number of basis vectors that are fixed. A basis vector, corresponding to an $m$-tuple of $v_1,\dots,v_N$, is fixed if and only if it is constant on each cycle of $\sigma$. Thus  the number of such vectors is $N$ to the number of cycles of $\sigma$, which is  $N$ to the number of prime factors of the polynomial.  The character of the sign representation is $(-1)^m$ times $(-1)$ to the number of cycles of $\sigma$. Altogether we deduce that the sum of this character is
\[\sum_{ \substack{f \in \mathbb F_q[x], ~\textrm{monic}\\
f~\textrm{squarefree}\\
 \deg(f) = m} }   ( (-1)^{n-1} N)^{\omega(f)} (-1)^{(n-1)m}, \] 
 where $\omega(f)$ is the number of prime factors of $f$. 
 But this is  equal to $(-1)^{(n-1)m}$ times the coefficient of $q^{-ms} $ in 
 \begin{align*}
  \sum_{ \substack{f \in \mathbb F_q[x], ~\textrm{monic}\\f ~\textrm{squarefree} }}   ( (-1)^{n-1} N)^{\omega(f)} q^{- {\operatorname{deg}(f) s}  }
  &= \prod_{ 
 \substack{
 g \in \mathbb F_q[x], ~\textrm{monic}\\g~\textrm{prime}}}  ( 1 + (-1)^{n-1 } N  q^{- \operatorname{deg}(g) s})\\
&=  \prod_{ \substack{g \in \mathbb F_q[x], ~\textrm{monic}\\ g~\textrm{prime}}} \prod_{k=1}^\infty  (1 - q^{-k \operatorname{deg}(g) s} )^{ -e_k(N)},
\end{align*}
by Definition~\ref{def:e}.
But we recognise that the right hand side is equal to 
\begin{align*}
 \prod_{k=1}^{\infty}  \zeta_{\mathbb F_q[t]} ( ks) ^{e_k(N)}  = \prod_{k=1}^\infty (1 - q^{1-ks})^{-e_k(N)}.
 \end{align*}
 Taking $U = q^{-s}$, we observe that the character sum is $(-1)^{nm-m}$ times the coefficient of $U^m$ in $\prod_{k=1}^\infty (1- q U^k)$, as claimed.
\end{proof}

We may  simplify the  formula in Lemma~\ref{homotopy-pconf} by 
introducing a sum over $m$, as follows.

\begin{cor}\label{homotopy-algebraic} Let $j,n,d,N\in \NN$, 
such that   $n>3$ and $d(n-3) \geq  j$, 
and let $V$ be a vector space of dimension $N$.
Then  
$$
\sum_{m=0}^d  \dim  (H^{mn - m - j  }_c (\Pconf_m,\Ql) \otimes V^{\otimes m}  \otimes \sgn^{n-1} )^{S_m} $$ 
is the coefficient of $q^{-j}$ in $\prod_{k=1}^\infty (1 - (-q)^{1 - k (n-2)}  )^{-e_k(N)}$. \end{cor}

\begin{proof}  Lemma~\ref{homotopy-pconf} implies that 
$\dim  (H^{mn - m - j  }_c (\Pconf_m,\Ql) \otimes V^{\otimes m}  \otimes \sgn^{n-1} )^{S_m} $ is   equal to $(-1)^{2mn-2m - j}= (-1)^j $ times the coefficient of $q^{mn-2m-j} U^m$ in the infinite product $\prod_{k=1}^\infty (1 - q U^k )^{-e_k(N)}$.  This infinite product has a power series  expansion 
\begin{equation}\label{eq:U}
\prod_{k=1}^\infty (1 - q U^k )^{-e_k(N)}=
\sum_{m=0}^\infty  c_m q^{mn-2m-j} U^m,
\end{equation}
for appropriate coefficients $c_m$. Our assumption that $d(n-3)\geq j$ ensures that only $m\leq d$ occur in this sum, 
since there are  no monomials where the power of $q$ is greater than the power of $U$ appearing. We have therefore shown that
$$
\sum_{m=0}^d  \dim  (H^{mn - m - j  }_c (\Pconf_m,\Ql) \otimes V^{\otimes m}  \otimes \sgn^{n-1} )^{S_m} =(-1)^j \sum_{m=0}^d  c_m.
$$
To calculate this we evaluate \eqref{eq:U} at $U = q^{2-n}$.
Replacing the variable $q$ with $-q$ removes the factor of $(-1)^j$ and so completes the proof.  
\end{proof}

Our next task is to show that  precisely the same 
power series occurs in the context of the double loop space.

\begin{lemma}\label{homotopy-topological} Assume $n \geq 2$. Let $X$ be the smooth vanishing set in $\mathbb C^n$ of a polynomial $f$ whose leading terms define a smooth hypersurface in projective space.   Let $\Omega^2 X$ be the (based) double loop space of $X$.
Then $\dim H^j(\Omega^2 X,\mathbb Q) $ is the coefficient of $q^{-j}$ in $\prod_{k=1}^\infty (1 - (-q)^{1 - k (n-2)}  )^{-e_k(N)}$, where
 $N $ is the dimension of $ H^{n-1}(X,\mathbb Q)$.
\end{lemma} 

\begin{proof}   
First we will show that the homotopy group $\pi_{k(n-2)+1}(X) \otimes \mathbb Q$ has dimension $(-1)^{k(n-2)-1} e_k(N)$ for all $k$, and that all other rational homotopy groups of $X$ vanish. We will then use these homotopy groups to calculate the cohomology of the double loop space. Observe that $X$ is homotopic to $\wedge^N S^{n-1}$, which follows from  \cite[Thm.~2]{Broughton} once we check that the polynomial $f$ defining $X$ is ``tame''. But this follows since its leading terms define a smooth hypersurface, whence the  partial derivatives of its leading terms have no common zero outside the origin. Thus  everywhere far from the origin at least one of the partial derivatives is large, which is precisely the criterion of tameness. 

The homotopy groups of $\wedge^N S^{n-1}$ were calculated by Hilton
\cite[Cor.~4.10]{Hilton},
with the outcome that 
$$
\pi_i ( \wedge^N S^{n-1}) = \sum_w  \pi_i(S^{(n-2)w + 1} )^{\frac{1}{w}  \sum_{d\mid w}  N^{w/d} \mu(d)} . 
$$
For rational homotopy groups, 
only $\pi_m$ is non-vanishing 
for odd dimensional spheres $S^m$.
In particular, for $n$ even we get 
$$
\pi_{ k(n-2)+1}(\wedge^N S^{n-1}) \otimes \mathbb Q= \mathbb Q^{ \frac{1}{k}  \sum_{d\mid k}  N^{k/d} \mu(d)},
$$ 
which has dimension 
$-e_k(N)=(-1)^{k(n-2)-1} e_k(N)$ by  \eqref{eq:inversion-formula}.
For  spheres of even dimension, both $\pi_m$ and $\pi_{2m-1}$ have one-dimensional rational homotopy groups. Hence for $n$ odd 
we have
\begin{align*}
\dim_{\mathbb Q} \pi_{k(n-2)+1} (\wedge^N S^{n-1}) \otimes \QQ
=  \frac{1}{k}  \sum_{d\mid k}  N^{k/d} \mu(d)+ 1_{k \equiv 2  \bmod{4}} \frac{2}{k}  \sum_{d\mid k/2}  N^{k/2d} \mu(d).
 \end{align*}
If $k$ is odd then an inspection of 
\eqref{eq:inversion-formula} reveals that the right hand side is equal to 
$e_k(N)=(-1)^{k(n-2)-1} e_k(N)$. If $k$ is even we write the right hand side as 
\begin{align*}
\frac{1}{k}  \sum_{d\mid k}  N^{k/d} \mu(d)+ 1_{k \equiv 2 \bmod{4}} \frac{2}{k}  \sum_{\substack{d\mid k, ~2\mid d}}  N^{k/d} \mu(\tfrac{d}{2}) &= \frac{1}{k}  \sum_{d\mid k}  N^{k/d} \mu(d)- \frac{2}{k}  \sum_{d\mid k, ~2\nmid k/d }  N^{k/d} \mu(d) \\
 &= \frac{1}{k}  \sum_{d|k}  (-N)^{k/d} \mu(d) ,
 \end{align*}
which again matches the formula for $(-1)^{k(n-2)-1} e_k(N)$.

The rational cohomology algebra on $H^*(\Omega^2, \mathbb Q)$ is the free graded commutative algebra on a basis for the rational homotopy of $X$, shifted by two degrees by  \cite[p.~311]{Sullivan}. 
Thus  $H^* (\Omega^2 X, \mathbb Q)$ is the free graded commutative algebra on $(-1)^{k(n-2)-1} e_k(N)$ generators in degree $k(n-2)-1$ for each $k$.

We wish to calculate the generating function $\sum_j \dim H^j(\Omega^2 X,\mathbb Q)  q^{-j} $. The generating function of the graded commutative algebra on one generator in degree $d$ is $(1 + q^{-d})$ if $d$ is  odd and $(1- q^{-d} )^{-1}$  if $d$ is  even.  Since  free products of algebras correspond to products of generating functions, the generating function of the cohomology algebra of $\Omega^2 X$ is $\prod_{k=1}^\infty (1 - (-q)^{1 - k (n-2)}  )^{-e_k(N)}$.
\end{proof} 

We may finally relate
Conjecture~\ref{stabilisation-conjecture} to our conjecture that 
 the spectral sequence in Theorem~\ref{main2} degenerates on the first page for 
$$m + s>- 4 \left( \left \lfloor \frac{d}{k-1} \right \rfloor \left( \frac{n}{2^k} - k + 1\right) - 1\right) .
$$ 
If the spectral sequence fails to degenerate, then some non-zero differential exists, and thus the dimension of $E^{m,s}_{\infty}$ is less than the dimension of $E^{m,s}_1$. To check that the spectral sequence degenerates, it is therefore sufficient to check that 
\begin{align*}
\dim H^{i+2d(n-k)} ( \Mor_{d,P}(\mathbb A^1,X), \mathbb Q)
=
\sum_{m+s = i} \dim E^{m,s}_1,
\end{align*}
since the left hand side is equal to 
$\dim H^{i+2d(n-k)} ( \Mor_{d,P}(\mathbb A^1,X), \mathbb Q_\ell)$.
   Under Conjecture~\ref{stabilisation-conjecture}, 
 we conclude from Lemma~\ref{lem:homotopic} that 
$$
\sum_{m+s = i} \dim E^{m,s}_1=
\dim H^{-i} ( Hom_{d,P} (\mathbb C, X), \mathbb Q)= H^{-i} ( \Omega^2 X, \mathbb Q).
$$
Let  $N=\dim H^{n-1}_c(X, \Ql)$. Then, by Poincar\'e duality and the universal coefficient theorem, we also have  $N=\dim H^{n-1}(X,\mathbb Q)$. (Note that we are also implicitly using the comparison of \'{e}tale and singular cohomology.) 
We now appeal to   the formulae of Corollary~\ref{homotopy-algebraic} and Lemma~\ref{homotopy-topological}, with the outcome that 
\begin{align*}
\dim H^j( \Omega^2 X , \mathbb Q) 
=  \sum_{m=0}^d \dim  (H^{mn - m - j  }_c (\Pconf_m,\Ql) \otimes H^{n-1}_c(X,\Ql)^{\otimes m}  \otimes \sgn^{n-1} )^{S_m}  .\end{align*}
We can take $j < 4 \left( \left \lfloor \frac{d}{k-1} \right \rfloor \left( \frac{n}{2^k} - k + 1\right) - 1\right)   \leq   d(n-3)$ to check the condition of Corollary~\ref{homotopy-algebraic}.   
We summarise our findings in the following result.

\begin{theorem}\label{t:top}
Conjecture~\ref{stabilisation-conjecture} implies 
Conjecture~\ref{con:degenerate} when $K=\CC$.
\end{theorem}

\end{document}